\documentclass[final,onefignum,onetabnum]{siamart1116}

\usepackage{geometry}
\usepackage[mathscr]{euscript}
\usepackage{color,amsmath,amssymb,mathrsfs}
\usepackage{graphicx}
\usepackage{xspace}
\usepackage{algorithmic,algorithm}
\usepackage{tikz}
\usepackage{pgfplots}
\pgfplotsset{compat=1.18}
\usepackage{upgreek}
\usepackage{showkeys,cite}
\usepackage{multirow}
\usepackage{enumitem} 
\usepackage{yfonts}
\usepackage{mathtools}
\usepackage[normalem]{ulem}
\usepackage{latexsym}
\usepackage{subcaption}
\DeclareMathAlphabet{\mathpzc}{OT1}{pzc}{m}{it}
\newcommand{\bs}[1]{\ensuremath{\boldsymbol{#1}}}

\usepackage{url} 
\usepackage{tikz}
\usepackage{comment}

\newcommand{\bvecS}[1]{\ensuremath{\boldsymbol{#1}}}
\newcommand{\M}[1]{\mathbf{#1}}
\newcommand{\norm}[2][]{\left\Vert #2\right\Vert_{#1}} 
\newcommand{\range}{\text{range}}
\newcommand{\eg}{e.g.\xspace}
\newcommand{\ie}{i.e.\xspace}

\newcommand{\both}{v} 
\newcommand{\bothdisc}{\textbf{\both}} 
\newcommand{\m}{m} 
\newcommand{\mdisc}{\textbf{m}} 
\newcommand{\msub}{m} 
\newcommand{\aux}{\xi} 
\newcommand{\auxdisc}{\boldsymbol{\aux}} 
\newcommand{\noise}{\M{e}} 
\newcommand{\noisesub}{e}
\newcommand{\data}{\M{d}} 
\newcommand{\datasub}{d}
\newcommand{\afwd}{\mathcal{G}} 
\newcommand{\afwddisc}{\M{G}} 
\newcommand{\sfwd}{\mathcal{F}} 
\newcommand{\sfwddisc}{\M{F}} 
\newcommand{\zero}{\mathsf{O}} 
\newcommand{\sjact}{\tilde{\sfwd}} 
\newcommand{\eps}{\bvecS{\varepsilon}} 
\newcommand{\epssub}{\varepsilon}
\newcommand{\tot}{\bvecS{\eta}} 
\newcommand{\totsub}{\eta} 

\newcommand{\obsOp}{\mathcal{V}} 
\newcommand{\sens}{\M{s}} 
\newcommand{\weight}{\M{w}} 
\newcommand{\Weight}{\M{W}} 

\newcommand{\nSens}{s} 
\newcommand{\nTimes}{\tau} 
\newcommand{\nData}{d} 
\newcommand{\nChosen}{k} 
\newcommand{\nChosenObs}{r} 
\newcommand{\Npatches}{N}
\newcommand{\nboth}{n} 
\newcommand{\nm}{n_m} 
\newcommand{\naux}{n_{\aux}} 

\newcommand{\bothSpace}{\mathscr{H}} 
\newcommand{\mSpace}{\mathscr{M}} 
\newcommand{\measSpace}{\mathscr{D}} 
\newcommand{\auxSpace}{\mathscr{X}} 

\newcommand{\mupr}{\mu_{\both}} 
\newcommand{\likelihood}[1]{\pi^{#1}_{\rm{like}}(\data\vert \both)} 
\newcommand{\mupostmW}[1]{\mu^{#1}_{\m\vert\datasub,w}} 
\newcommand{\mupostboth}[1]{\mu^{#1}_{\both\vert\datasub}} 
\newcommand{\mupostbothW}[1]{\mu^{#1}_{\both\vert\datasub,w}} 
\newcommand{\mprdisc}{\mdisc_0} 
\newcommand{\Gprdisc}{\boldsymbol{\Gamma}_{\mdisc}}

\newcommand{\bothmeandisc}{\bar{\bothdisc}} 
\newcommand{\meanm}{\m_0} 
\newcommand{\meanboth}{\both_0} 
\newcommand{\meaneps}{\eps_0} 
\newcommand{\meannoise}{\noise_0} 
\newcommand{\meantot}{\tot_{0|\both}} 
\newcommand{\Cbothboth}{\mathcal{C}_{\both\both}} 
\newcommand{\Cbotheps}{\mathcal{C}_{\both\epssub}} 
\newcommand{\Cepsboth}{\mathcal{C}_{\epssub\both}} 
\newcommand{\Cepseps}{\M{\Gamma}_{\epssub\epssub}} 
\newcommand{\Cnoise}{\M{\Gamma_{\noisesub}}} 
\newcommand{\Cnoiseinv}{\Cnoise^{\!\!-1}}

\newcommand{\Cepsgboth}{\M{\Gamma}_{\epssub\vert\both}}
\newcommand{\Ctotgboth}{\M{\Gamma}_{\totsub\vert\both}} 

\newcommand{\meantotgboth}{\tot_{0\vert\both}} %
\newcommand{\meanpostboth}[1]{\both^{#1}_{0\vert\datasub}} 
\newcommand{\meanpostbothW}[1]{\both^{#1}_{0\vert\datasub,w}} 
\newcommand{\Cpostboth}[1]{\mathcal{C}^{#1}_{\both\vert\datasub}} 
\newcommand{\CpostbothW}[1]{\mathcal{C}^{#1}_{\both\vert w}} 
\newcommand{\CpostbothWdisc}[1]{\M{C}^{#1}_{\both \vert \datasub,w}} 
\newcommand{\CpostmW}[1]{\mathcal{C}^{#1}_{\m\vert\datasub,w}} 

\DeclareMathOperator*{\argmin}{arg\,min}
\DeclareMathOperator*{\argmax}{arg\,max}
\DeclareMathOperator*{\trace}{trace}

\newtheorem{thm}{Theorem}
\newtheorem{rmk}{Remark}

\newtheorem{cor}{Corollary}

\begin{document}
\setlength{\belowcaptionskip}{-11pt}

\setlength{\abovedisplayskip}{4pt}
\setlength{\belowdisplayskip}{4pt}

\title{Non-intrusive optimal experimental design for large-scale nonlinear
Bayesian inverse problems using a Bayesian approximation error
approach}

\def\addressC{Interdisciplinary Center for Scientific Computing (IWR), Heidelberg University, Heidelberg, Germany}
\def\addressD{Courant Institute of Mathematical Sciences, New York University, New York, NY, USA.}
\def\addressO{Department of Engineering Science, University of Auckland, Auckland, New Zealand}

\author{Karina Koval\footnotemark[1]
    \and Ruanui Nicholson\footnotemark[3]}
\renewcommand{\thefootnote}{\fnsymbol{footnote}}

\footnotetext[1]{\addressC\ (\email{karina.koval@iwr.uni-heidelberg.de}).}
\footnotetext[3]{\addressO\ (\email{ruanui.nicholson@auckland.ac.nz}).}

\maketitle

\begin{abstract}
We consider optimal experimental design (OED) for nonlinear inverse problems within the Bayesian framework. Optimizing the data acquisition process for large-scale \emph{nonlinear} Bayesian inverse problems is a computationally challenging task since the posterior is typically intractable and commonly-encountered optimality criteria depend on the observed data. Since these challenges are not present in OED for \emph{linear} Bayesian inverse problems, we propose an approach based on first linearizing the associated forward problem and then optimizing the experimental design. Replacing an accurate but costly model with some linear surrogate, while justified for certain problems, can lead to incorrect posteriors and sub-optimal designs if model discrepancy is ignored. To avoid this, we use the Bayesian approximation error (BAE) approach to formulate an A-optimal design objective for sensor selection that is aware of the model error. In line with recent developments, we prove that this uncertainty-aware objective is  independent of the exact choice of linearization. This key observation facilitates the formulation of an uncertainty-aware OED objective function using a completely trivial linear map, the zero map, as a surrogate to the forward dynamics. The base methodology is also extended to marginalized OED problems, accommodating uncertainties arising from both linear approximations and unknown auxiliary parameters. Our approach only requires parameter and data sample pairs, hence it is particularly well-suited for black box forward models. We demonstrate the effectiveness of our method for finding optimal designs in an idealized subsurface flow inverse problem and for tsunami detection. 
\end{abstract} 



\section{Introduction}\label{sec1}
Many problems of interest in engineering and the natural sciences can be described by mathematical models involving partial differential equations (PDEs) or systems of ordinary differential equations. In various settings, however, it is parameters of the models, such as coefficients or boundary/initial conditions, which are of importance. These parameters often cannot be observed directly, and are instead estimated using observations of a related quantity as well as the governing mathematical equations. In typical applications, the observations are noisy and informative about a small subset of the unknown parameters, making parameter estimation an ill-posed problem. The Bayesian approach to parameter estimation (or \emph{inverse problems}) is commonly used to deal with the ill-posedness of the problem. In the Bayesian approach, the solution to the inverse problem is a conditional distribution that enables quantifying the probability that the observed data originated from any candidate parameter choice. 

The quality and quantity of the measurement data has a significant effect on the quality of the solution to the Bayesian inverse problem. Thus, it is crucial to guide data acquisition and choose experimental conditions that lead to informative observations. This requires solving an optimal experimental design (OED) problem~\cite{alexanderian:oedreview,chaloner:oed}. While the OED framework encompasses various different problem-specific design types, the subclass we focus on is that of sensor placement. Specifically, assuming measurements can be acquired at some set of sensors, the OED problem consists of choosing the most informative sensor subset from some candidate set.  

Finding the optimal sensor placement is a particularly challenging problem if the parameter-to-observable (PTO) map is nonlinear. For nonlinear Bayesian inverse problems, the most informative design can be formulated as the optimizer of an expected utility function that evaluates the effectiveness of solving the inverse problem with data collected at any design choice. However, finding the analytic optimizer is generally impossible, and numerical approximation presents further difficulties. These stem from various factors --- many utility functions (including the classical A- and D-optimality) are functions of the intractable posterior distribution, the expectation is taken with respect to the (typically unknown) conditional density for the data given the design (\ie, the \emph{evidence}), and even approximate evaluation of the utility function requires a large number of costly forward (and potentially adjoint) PDE solves.

A commonly employed approximation to the expected utility function is obtained via linearization of the PTO map or through a Gaussian approximation to the posterior~\cite{alexanderian:nonlinoed,alexanderian:baeoed,wu:fast}. This is also the approach we take in this present work, however, we employ a \emph{global} linearization (as opposed to \emph{local} linearizations as in the aforementioned works) and incorporate the resulting model error into the Bayesian inverse problem using the Bayesian approximation error (BAE) approach~\cite{kaipio2013approximate,kaipio:invCrimes}. Focusing on the A-optimal criterion, we expand on the work of~\cite{nicholson:linearBAE} and show that the optimal designs obtained using the global linear map are independent of the choice of linearization when model error is incorporated with the BAE approach. As a consequence, we present a scalable, derivative-free linearization approach for finding A-optimal sensor placements for nonlinear Bayesian inverse problems. We also extend our approach to OED under uncertainty, \ie, finding designs that are optimal for estimating some primary parameter(s) of interest while incorporating uncertainty in auxiliary nuisance parameters. 

\paragraph{Related work}
There has been a recent surge of interest in optimal experimental design for parameter estimation problems governed by uncertain mathematical models. Many of these references assume exact knowledge of the governing dynamics and focus on uncertainty due to unknown nuisance/auxiliary inputs to the model~\cite{alexanderian:marginal,alexanderian:baeoed,feng:OEDUU,koval:oed}, or misspecification of the statistical model~\cite{attia:robustOED}. However, in certain situations, the accurate forward model is unknown or prohibitively expensive to use in an optimization procedure, hence an approximate or surrogate model is used instead. The importance of accounting for model error in the context of Bayesian inversion has been well-established in the literature (see, \eg,~\cite{kaipio:stats,kennedy:calibration}). This importance naturally extends to the OED problem, and in~\cite{cvetkovic:mitigating}, the authors explore OED under uncertainty from the perspective of designing observation operators that mitigate model error. In contrast, we focus on the classical OED problem of choosing experimental conditions that lead to optimal parameter inference and formulate an uncertainty-aware optimality criterion.   

Model error-informed approximations to posterior distributions can be obtained in various ways, \eg, by employing Gaussian processes~\cite{kennedy:calibration,higdon:calibration}, via error-corrected delayed acceptance Markov chain Monte Carlo (MCMC)~\cite{Cui:MCMC}, and through the Bayesian approximation error approach~\cite{kaipio:invCrimes}. Here, we follow the latter technique for incorporating model error into the OED problem. The BAE approach is also used for sensor placement design problems in~\cite{alexanderian:baeoed} to account for model error stemming from unknown nuisance parameters, however our formulation differs in the following key ways: (i) we use a global linear approximation to the parameter-to-observable map, rather than many local Gaussian approximations, (ii) our approach accommodates model error due to the use of a linear surrogate as well as auxiliary input parameters, and (iii) we employ the {\em full} BAE error approach, rather than the so-called \emph{enhanced error model}. The latter point facilitates a derivative-free approach to linearized OED. An alternative derivative-free approach to OED under model uncertainty is presented in~\cite{dunbar:ensemble}, where D-optimal designs are obtained using an ensemble-based approximation to the posterior covariance matrix. In their approach, which is based on a calibrate-emulate-sample algorithm, a cheap-to-evaluate emulator and the corresponding model error correction term are learned using Gaussian processes. 
\paragraph{Contributions}
The main contributes of this article are: 
(i) We propose a tractable, uncertainty-aware approach for optimal sensor placement in infinite-dimensional Bayesian inverse problems by substituting the costly accurate forward model with an inexpensive linear surrogate and incorporating the corresponding model error using the BAE approach.  
(ii) Expanding on the work of~\cite{nicholson:linearBAE}, we show that the resulting approximate uncertainty-aware OED objective is asymptotically independent of the specific choice of linear surrogate. This result is key in formulating a black-box OED algorithm that only requires sample parameter and data pairs. 
(iii) The base methodology and theorems are also extended to marginalized OED, where the unknowns include some primary parameters of interest as well as some unimportant (or uninteresting) auxiliary inputs.
(iv) With two model problems, we present comprehensive computational studies that illustrate the effectiveness of our approach for both standard and marginalized OED. In particular, we showcase the necessity of incorporating model error into the design problem, and the effectiveness of the optimal designs computed using our non-intrusive approach in solving the original high-fidelity inverse problem. 

\paragraph{Limitations} 
Of course, the approach presented also has some limitations, which include: 
	(i) Similar to any sort of linearization-based approximation procedure, our approach may not be well-suited for Bayesian inverse problems involving highly nonlinear PTO maps or non-Gaussian priors. While our approach does not place explicit restrictions on the prior, there is no guarantee that the resulting designs are effective in solving the original problem. However, if a the problem is not well-approximated by a Gaussian, the uncertainty-informed approximate posterior can still be used in more accurate (hence more costly) OED schemes as, \eg, a proposal density for MCMC-based methods, or as a reference density for transport-based approaches.
(ii) Independence of the approximate posterior to the choice of linear surrogate is only guaranteed asymptotically, thus sufficient approximation to the second-order statistics may require a potentially large number of expensive PDE solves, particularly if the linear surrogate is a poor approximation to the true dynamics (as is the case when using the zero operator). This challenge is not exclusive to our methodology, since many OED algorithms rely on expensive offline PDE solves. However, if the number of samples required is prohibitively expensive and the PTO map is differentiable, using the derivative as a control variate could speed up convergence to the true error statistics. A theoretical study of the surrogate-dependent statistical convergence could help develop a better understanding of the suitability and limitations of our approach.
(iii) In the current paper the methods used are based on explicit formulation of the associated covariance operators/matrices, and as such could be become infeasible for large problems. However, low-rank and other matrix representations, or potentially matrix-free variants,  could avoid this potential bottleneck.

\paragraph{Outline of the paper}
In~\Cref{sec2} we introduce notation used throughout the paper and outline relevant background material on infinite-dimensional Bayesian inverse problems, and the Bayesian approximation error approach. In~\Cref{sec3}, we outline our uncertainty-aware linearize-then-optimize approach to optimal experimental design and show that the resulting objective is independent of the choice of linear surrogate. As shown in~\Cref{sec3 3}, this independence also extends to marginalized OED under model error, \ie, in this section, we consider OED for Bayesian inverse problems where the model error stems from the use of a linear surrogate and the presence of auxiliary unknown inputs to the forward model. The effectiveness of our proposed method is illustrated in~\Cref{sec4} and~\Cref{sec5}, where OED is considered for an idealized subsurface flow geometric inverse problem and a tsunami source detection problem, respectively.

\section{Background}\label{sec2}

In this section, we motivate our approach for carrying out derivative-free linearized OED for large-scale PDE-constrained problems. We begin by introducing the  required notation and preliminaries in Section \ref{sec2 1}.  We
then briefly recall the Bayesian approach to inverse problems and the Bayesian  approximation error approach.

\subsection{Preliminaries}\label{sec2 1}
In this article we are interested in parameters which take values in possibly infinite-dimensional Hilbert spaces. For a Hilbert space $\bothSpace$, the corresponding inner product is denoted by $\langle \cdot,\cdot \rangle_{\bothSpace}$ and the associated norm by $\norm[\bothSpace]{\cdot}$. For $\bothSpace_1$ and  $\bothSpace_2$ Hilbert spaces, we let $\mathcal{L}(\bothSpace_1,\bothSpace_2)$ denote the space of bounded linear operators from $\bothSpace_1$ to $\bothSpace_2$. 

In this article we are particularly interested in Gaussian measures on Hilbert spaces. To this end, throughout the article we use $\mathcal{N}(\meanboth,\mathcal{C}_{\both\both})$ to denote a Gaussian measure with mean $\meanboth$ and covariance operator $\mathcal{C}_{\both\both}$ . In the infinite dimensional case, the covariance
operator is required to satisfy certain regularity assumptions to ensure the Bayesian inverse problem is
well-defined \cite{stuart:bayes}. As such, we assume $\mathcal{C}_{\both\both}$ is a strictly positive
self-adjoint trace-class operator. We recall the Gaussian measure $\mathcal{N}(\meanboth,\mathcal{C}_{\both\both})$ on $\bothSpace$ induces the \emph{Cameron-Martin} space  $\mathscr{E} \coloneqq \range(\Cbothboth^{1/2})$, which is endowed with the inner product $\langle u,w \rangle_{\mathscr{E}} \coloneqq \langle \Cbothboth^{-1/2}u, \Cbothboth^{-1/2}w\rangle_{\bothSpace}$ for all $u,w \in \mathscr{E}$.

For $\bothSpace_1$ and $\bothSpace_2$ real Hilbert spaces, and for a linear operator $\mathcal{A}\in\mathcal{L}(\bothSpace_1,\bothSpace_2)$, we denote by $\mathcal{A}^\ast\in\mathcal{L}(\bothSpace_2,\bothSpace_1)$ the adjoint. Moreover, for $u\in\bothSpace_1$ and $w\in\bothSpace_2$, we define the (outer product) operator $u\otimes w$ by $(u\otimes w)v=\langle w,v \rangle_{\mathcal{H}_2} u$, for any $v\in\bothSpace_2$. The (cross-)covariance operator (between $u$ and $w$) is then defined
\begin{equation}
\mathcal{C}_{uw}=\mathbb{E}[(u-u_0)\otimes(w-w_0)]=\mathcal{C}_{wu}^\ast.
\label{eq:crossCovs}
\end{equation}

\subsection{Bayesian inverse problems}\label{sec2 2}

The discussion presented here is carried out in the possibly infinite-dimensional Hilbert space setting~\cite{stuart:bayes}. We consider the problem of inferring an unknown parameter $\both \in \bothSpace$ given observations $\data \in \mathbb{R}^{\nData}$ that are related to $\both$ as 
\begin{equation}\label{eq: main1}
\data = \afwd(\both) + \noise, 
\end{equation}
where $\noise \in \mathbb{R}^{\nData}$ denotes measurement noise. 
In our target applications, the (potentially) nonlinear parameter-to-observable map $\afwd: \bothSpace \rightarrow \mathbb{R}^{\nData}$ can be decomposed into the composition of two operators, 
$\afwd(\cdot) \coloneqq (\obsOp \circ \mathcal{S})(\cdot)$, where the parameter-to-state map $\mathcal{S}: \bothSpace \rightarrow \mathcal{X}$ (for some suitably-chosen function space $\mathcal{X}$) requires solving a partial differential equation (PDE), and $\obsOp: \mathcal{X} \rightarrow \mathbb{R}^{\nData}$ denotes a linear spatio-temporal observation operator. 

In the Bayesian approach for parameter estimation, $\both$ is treated as a random variable and is endowed with a prior probability law ($\mupr$) that encompasses any knowledge we have about $\both$ prior to data collection. Here we use a Gaussian prior $\mathcal{N}(\meanboth,\Cbothboth)$ with prior mean $\meanboth \in \mathscr{E}$ and a trace-class covariance operator $\Cbothboth$.  In practice it is often assumed that the data is corrupted by Gaussian noise which is independent of the parameter, \ie, $\noise \sim \mathcal{N}(\meannoise, \Cnoise)$, where $\Cnoise \in \mathbb{R}^{\nData \times \nData}$ is a symmetric-positive-definite (SPD) covariance matrix and $\noise\perp \both$.  

Within the Bayesian framework, the solution to the inverse problem is the (data-informed) posterior law of $\both|\data$, which we denote by $\mupostboth{}$. The posterior law is established as a Radon-Nikodym derivative via Bayes' law 
and takes the following form under our additive Gaussian noise model, 
\begin{equation}
\frac{d\mupostboth{}}{d\mupr} \propto \likelihood{}, \quad \likelihood{} \propto \exp \left[- \frac{1}{2} \norm[\Cnoiseinv]{\data-\left(\afwd(\both)+\meannoise\right)}^2 \right].
\label{eq:RN_deriv}
\end{equation} 

Fully characterizing the posterior in the case of PDE-constrained inverse problems is generally infeasible. As a computationally practical alternative it is common to compute a Gaussian approximation to the posterior measure. A commonly utilized approximation is the so-called Laplace approximation~\cite{wong:asymptotic}, which approximates $\mu_{\both\vert\datasub}$ as a Gaussian, $\mathscr{L}_{\mupostboth{}} \coloneqq \mathcal{N}(\both_{\rm MAP},\mathcal{C}_{\rm post})$, centered around the maximum a posteriori (MAP) estimate~\cite{helin:map},
\begin{equation}
\both_{\rm MAP}\coloneqq\argmin_{\both\in\mathscr{E}}\mathcal{J}(\data,\both)\label{eq MAP},\quad \mathcal{J}(\M{d},\both) \coloneqq\frac{1}{2}\norm[\Cnoiseinv]{\data-\left(\afwd(\both)+\meannoise\right)}^2+\frac{1}{2}\norm[\Cbothboth^{-1}]{\both-\meanboth}^2.
\end{equation}
The approximate posterior covariance operator of the Laplace approximation is given by $\mathcal{C}_{\rm post}=\mathcal{H}^{-1}_{\both_{\rm MAP}}$, where $\mathcal{H}_{\both_{\rm MAP}}$ denotes the Hessian of $\mathcal{J}$ evaluated at the MAP estimate.

Construction of the Laplace approximation presents many challenges. While the MAP estimator is guaranteed to exist under relatively mild assumptions on the PTO map and data misfit, uniqueness of the estimator can not be guaranteed in general for nonlinear Bayesian inverse problems (see~\cite{dashti:map} for details). Additional theoretical difficulties arise if the map $\afwd$ is not differentiable since the Laplace approximation requires access to derivatives of the PTO map. Constructing the Laplace approximation is also computationally costly. Finding the MAP estimator via minimization of~\eqref{eq MAP} can involve many applications of the PTO map and its adjoint. Additionally, the inverse Hessian defining the posterior covariance operators typically not available in explicit form and needs to be estimated with additional forward and adjoint solves using, \eg, optimal low-rank approximations as described in~\cite{spantini:LRA}. This computational cost is exacerbated when one seeks optimal designs for nonlinear Bayesian inverse problems. As discussed in~\Cref{sec3}, OED in the nonlinear setting typically requires solving many Bayesian inverse problems. Thus, typical approaches (\eg,~\cite{alexanderian:nonlinoed}) involve building many Gaussian approximations.

An alternative approach to obtaining a Gaussian approximation to the posterior distribution that circumvents some of the aforementioned challenges is to employ a linear surrogate to the PTO map. We leave the precise form of the linear (affine) surrogate arbitrary for now, and denote the linear map with $\sfwd \in \mathcal{L}(\bothSpace,\mathbb{R}^{\nData})$. Naturally, replacing the nonlinear map with some linear approximation leads to errors that need to be accounted for to avoid erroneous overconfidence and bias in the resulting Gaussian posterior. These model approximation errors can be incorporated using the Bayesian approximation error (BAE) approach \cite{kaipio2013approximate,kaipio:invCrimes}. 

Before outlining the BAE approach to account for the resulting approximation errors, we  recall the explicit results of the linear Gaussian case.
\paragraph{The linear Gaussian case} 
In the case of a linear PTO map, \ie,
\begin{align}
\data=\sfwd\both+\noise\nonumber
\end{align}
for $\sfwd\in\mathcal{L}(\bothSpace,\mathbb{R}^{\nData})$, under Gaussian prior and additive Gaussian noise, the posterior measure is Gaussian, with conditional mean and conditional covariance given, respectively, by (see \cite[Example 6.23]{stuart:bayes})
\begin{align}\label{eq linGaus}
\meanpostboth{} &= \Cpostboth{} \left(\sfwd^\ast\Cnoiseinv(\data-\meannoise)+\Cbothboth^{-1}\meanboth\right),\quad
\Cpostboth{} =(\Cbothboth^{-1}+\sfwd^\ast\Cnoiseinv\sfwd)^{-1}.
\end{align}

\subsection{The Bayesian approximation error approach}\label{sec2 3}

Neglecting the model errors and uncertainties induced by the use of a surrogate to the PTO map in general leads to overconfidence in biased estimates \cite{kaipio:invCrimes,kaipio2013approximate}.
To avoid this, we employ the BAE approach, which has been used for incorporating various types of model uncertainties into the Bayesian inverse problem as well as the OED problem (see, \eg,~\cite{arridge:bae,nicholson:BAE,hanninen:qpat,alexanderian:baeoed,alexanderian:baeoed}). Here we outline the BAE approach as it pertains to our method.

As alluded to in the preceding section, the typical approach to OED within the Bayesian framework requires repeatedly solving the optimization problem  (\ref{eq MAP}) and constructing the (local) Laplace approximation\footnote{An alternative (approximate) approach is proposed in \cite{wu:fast} (see Section 3.4) where the Laplace approximation is carried out at samples from the prior, thus negating the need to compute any MAP estimates. However, this approximate approach is not well-suited for highly ill-posed inverse problems.}. While using a surrogate (or reduced order model approximation) to the PTO map can reduce the computational costs associated with computing the MAP estimate, the cost for accurate Monte Carlo (MC) approximation of the optimality criterion can still be significant since it requires approximating the MAP estimate for many data samples. 

In what follows, we let $\sfwd:\bothSpace\to\mathbb{R}^{\nData}$ denote a surrogate model. The starting point for the BAE approach is to rewrite the accurate relationship between the data and parameters (\ref{eq: main1}) as
\begin{equation}\label{eq:BAE1}
\begin{aligned}
\data&=\afwd(\both)+\noise=\sfwd(\both)+\bvecS{\eps}(\both)+\noise=\sfwd(\both)+\tot(\both),
\end{aligned}
\end{equation}
where $\mathbb{R}^{\nData}\ni\bvecS{\eps}(\both)\coloneqq\afwd(\both)-\sfwd(\both)$ is the {\em approximation error}, and $\mathbb{R}^{\nData}\ni\tot(\both)\coloneqq\bvecS{\eps}(\both)+\noise$ is the {\em total error}. At this point, a conditional Gaussian approximation for the approximation error is made, \ie,
\begin{align}\label{eq baeNorm}
\pi_{\eps\vert\both}(\eps\vert \both)\approx\mathcal{N}(\eps_{0\vert\both},\Cepsgboth) \sim \bvecS{\eps}(\both),
\end{align}
with (where we denote $\meaneps \coloneqq \mathbb{E}\eps$)
\begin{align}
\eps_{0\vert\both}=\meaneps+\Cepsboth\Cbothboth^{-1}(\both-\meanboth),\quad
    \Cepsgboth=\Cepseps-\Cepsboth\Cbothboth^{-1}\Cbotheps,\nonumber
    \label{eq:totErr_stats}
\end{align}
where $\Cbotheps$ and $\Cepsboth$ defined as in~\eqref{eq:crossCovs}. As a direct consequence of (\ref{eq baeNorm}) we have 
\begin{align}
\pi_{\tot\vert\both}(\tot\vert \both)\approx\mathcal{N}(\tot_{0\vert\both},\Ctotgboth),\quad\tot_{0\vert\both}=\meannoise+\eps_{0\vert\both},\quad \Ctotgboth=\Cnoise+\Cepsgboth
\end{align}
and we approximate Bayes' law (\ref{eq:RN_deriv}) by
\begin{equation}
\frac{d\mupostboth{}}{d\mupr} \propto \pi_{\tot\vert\both}(\tot\vert \both)\propto\exp\left\{-\frac{1}{2} \norm[\Ctotgboth^{-1}]{\data-\left(\sfwd(\both)+\tot_{0\vert\both}\right)}^2\right\}.
\label{eq:RN_derivBAE}
\end{equation} 
In some applications of the BAE the further approximation $\Cbotheps=0$ is employed, which is often referred to as {\em the enhanced error model}~\cite{arridge:bae,kaipio2013approximate}.

A particularly straight-forward choice of surrogate model is a (affine) linear model, \ie, take $\sfwd\in\mathcal{L}(\bothSpace,\mathbb{R}^{\nData})$. The standard choice for such a surrogate model is the (generalized) derivative of $\afwd$ evaluated at some nominal point $\both_\ast\in\bothSpace$, \ie, $\sfwd\both=\afwd(\both_\ast)+\mathsf{D}_\m\afwd(\both_\ast)(\both-\both_\ast)$. However, as long as the approximation error is taken into account using the BAE approach, the resulting (approximate) posterior is independent of the particular choice of linear surrogate, as stated in the following theorem from~\cite{nicholson:linearBAE}.

\begin{thm}\label{thm:1}
Let $\bothSpace$ be a Hilbert space with $v\in\bothSpace$ and assume $v$ has prior measure $\mu_v$ with mean and (trace-class) covariance operator given by $v_0$ and $\mathcal{C}_{\both\both}$, respectively. Suppose further that
\begin{align}
\data=\afwd(v)+\noise,\nonumber
\end{align}
where $\afwd:\bothSpace\rightarrow\mathbb{R}^{\nData}$ is the bounded PTO map, and $\noise\in\mathbb{R}^{\nData}$ has mean $\meannoise$ and covariance matrix $\Cnoise$.
Then the approximate likelihood model parameterized by $\sfwd\in\mathcal{L}(\bothSpace,\mathbb{R}^{\nData})$,
\begin{align}
\likelihood{\sfwd} \propto\exp{\left\{-\frac{1}{2}\norm[2]{L_{\eta\vert \both}(\data-\sfwd\both-\meantotgboth)}^2\right\}},\nonumber
\end{align}
where $L_{\eta\vert \both}^TL_{\eta\vert \both}=\Ctotgboth^{-1}$, and 
\begin{align}
\meantotgboth=\meannoise+\meaneps+\Cepsboth\Cbothboth^{-1}(\both-\meanboth), \quad \Ctotgboth=\Cnoise+\Cepseps-\Cepsboth\Cbothboth^{-1}\Cbotheps,\nonumber
\end{align}
is independent of the choice of $\sfwd$.
\end{thm}

As an immediate consequence, using (\ref{eq linGaus}), we have the following corollary.
\begin{cor}\label{cor:1}
The resulting (approximate) posterior is Gaussian and independent of the choice of linear model, with conditional mean and covariance given by
\begin{align}
\meanpostboth{}&=\Cpostboth{} (\sjact^*\Ctotgboth^{-1}(\data-\meannoise-\meaneps+\Cepsboth\Cbothboth^{-1}\meanboth)+\Cbothboth^{-1}\meanboth )\label{eq: BAEcm},\\
    \Cpostboth{} &=(\Cbothboth^{-1}+\sjact^*\Ctotgboth^{-1}\sjact)^{-1}\label{eq: BAEvar},
\end{align}
respectively, with $\mathcal{L}(\bothSpace,\mathbb{R}^{\nData})\ni\sjact\coloneqq\sfwd+\Cepsboth\Cbothboth^{-1}$.
\end{cor}
\begin{rmk}
   The independence of the approximate posterior to the choice of linearization is a consequence of employing a linear surrogate map and accounting for model uncertainty using the (full) BAE approach. Specifically, employing the BAE approach introduces an affine (in $\both$) correction to the approximate forward model, namely $\meaneps+\Cepsboth\Cbothboth^{-1}(\both-\meanboth)$. As such, any change in the linear surrogate map is canceled out/offset by the correction term. However, as seen in the proof of~\cite[Theorem 1]{nicholson:linearBAE}, the cross-covariance term $\Cepsboth$ is crucial for the result. If the cross-covariance is neglected (\eg, using the enhanced error model) the result would not follow.
\end{rmk}
In various inverse problems there are additional uncertain model parameters which are not estimated~\cite{kaipio2013approximate}. This is usually because $a$) these parameters are not of interest, or $b$) estimation of these parameters is too costly (or impossible). We will refer to these additional (to the primary parameters of interest) parameters as {\em auxiliary parameters}, though in the literature the terms nuisance parameters, secondary parameters, and latent parameters are also common. Neglecting the uncertainty in the auxiliary parameters while inferring the primary parameters typically results in misleading estimates and significantly underestimated uncertainty. The same can also be said for the OED process; neglecting the uncertainty in the auxiliary parameters while carrying out OED generally yields sub-optimal designs \cite{alexanderian:baeoed}.

The BAE approach can be applied as a means to account for uncertainty in auxiliary parameters during both the inference and OED stages. Specifically, let $\both = (\m,\aux)$, with $\m\in\mSpace$ the primary parameter $\aux\in\auxSpace$ the auxiliary parameter. Then assuming the accurate PTO model is $\afwd:\mSpace\times \auxSpace\to\mathbb{R}^\nData$, we can rewrite the relationship between the  data and the primary parameters as (cf. (\ref{eq:BAE1}))
\begin{equation}\label{eq:BAE2}
\begin{aligned}
\data&=\afwd(\m,\aux)+\noise =\sfwd(\m)+\bvecS{\eps}(\m,\aux)+\noise =\sfwd(\m)+\tot(\m,\aux).
\end{aligned}
\end{equation}

In keeping with the theme of the paper, we will be particularly interested in cases where the surrogate model is linear, \ie, $\sfwd\in\mathcal{L}(\mSpace,\mathbb{R}^{\nData})$. For such cases the following corollary, outlining the independence of the approximate posterior to the choice of linear(-ized) model can be useful.   

\begin{cor}
Let $\mSpace$ and $\auxSpace$ be Hilbert spaces with $m\in\mSpace$ and $\aux\in\auxSpace$. Assume $(\m,\aux)$ has (joint) prior measure $\mu_{\m,\aux}$ with mean and (trace-class) covariance operator given by $(\m_0,\aux_0)$ and $\mathcal{C}_{\m,\aux}=\begin{pmatrix}\mathcal{C}_{\m\m} & \mathcal{C}_{\m\aux}\\ \mathcal{C}_{\aux\m} & \mathcal{C}_{\aux\aux} \end{pmatrix}$, respectively. Suppose further that
\begin{align}
\data=\afwd(\m,\aux)+\noise,\nonumber
\end{align}
where $\afwd:\mSpace\times\auxSpace\rightarrow\mathbb{R}^{\nData}$ is the bounded PTO map, and $\noise\in\mathbb{R}^{\nData}$ has mean $\meannoise$ and covariance matrix $\Cnoise$.
Then the approximate Gaussian marginal posterior measure parameterized by $\sfwd\in\mathcal{L}(\mSpace,\mathbb{R}^{\nData})$, with mean and covariance given by 
\begin{align}
\m_{0\vert\datasub}^{\sfwd} = \mathcal{C}_{\m\vert\datasub}\left(\sjact^*\Gamma_{\totsub\vert\m}(\data-\meannoise-\meantot+\mathcal{C}_{\epssub\m}\mathcal{C}_{\m\m}^{-1}\meanm)+\mathcal{C}_{\m\m}^{-1}\meanm\right),\quad
\mathcal{C}_{\m\vert\totsub}^{\sfwd} =(\mathcal{C}_{\m\m}^{-1}+\sjact^*\mathcal{C}_{\totsub\vert\m}\sjact)^{-1},\nonumber
\end{align}
respectively, with $\mathcal{L}(\mSpace,\mathbb{R}^{\nData})\ni\sjact\coloneqq\sfwd+\mathcal{C}_{\totsub\m}\mathcal{C}_{\m\m}^{-1}$, is independent of the choice of $\sfwd$.
\end{cor}

\paragraph{Computing the approximation error statistics} In general the (second order) statistics of the approximation error are not known a priori. As such, these are computed using Monte Carlo simulations. More specifically, an ensemble of samples $\both^{(\ell)}=(\m^{(\ell)},\aux^{(\ell)})$ for $\ell=1,2,\dots,q$ are drawn from the joint prior and the associated approximation error is computed, \ie, 
\begin{align}
\eps^{(\ell)}=\afwd(\both^{(\ell)})-\sfwd\both^{(\ell)}.
\label{eq:approxErr}
\end{align}
From this ensemble of samples the sample means and (cross-)covariances can be computed:
\begin{align}\label{eq epsStats1}
\meaneps&\approx\frac{1}{q}\sum_{\ell=1}^q\eps^{(\ell)},\quad
\Cepseps\approx\frac{1}{q-1}\sum_{\ell=1}^q(\eps^{(\ell)}-\meaneps)\otimes(\eps^{(\ell)}-\meaneps),\\
\label{eq epsStats2}\Cepsboth&\approx\frac{1}{q-1}\sum_{\ell=1}^q(\eps^{(\ell)}-\meaneps)\otimes(\both^{(\ell)}-\meanboth),\\ \label{eq epsStats3}\meanboth&\approx\frac{1}{q}\sum_{\ell=1}^q\both^{(\ell)},\quad \Cbothboth\approx \frac{1}{q-1}\sum_{\ell=1}^q(\both^{(\ell)}-\meanboth)\otimes(\both^{(\ell)}-\meanboth).
\end{align}

As noted previously, the number of samples required to compute the required quantities in (\ref{eq epsStats1}) and (\ref{eq epsStats2}) could be reduced by employing a control variate approach~\cite[Chapter 2.3]{robert1999monte} with a judicious choice of control variate such as the (generalized) derivative.

\section{A data-driven approach to optimal experimental designs}\label{sec3}

Here we present our tractable approach to approximating optimal designs for nonlinear Bayesian inverse problems. In~\cref{sec3 1} we outline relevant material on sensor placement design problems and their associated computational challenges. 
Our black-box approach that overcomes these challenges for OED and marginalized OED is presented in~\cref{sec3 2}. 
In~\cref{subsec:discreteoed} we discuss an efficient numerical implementation of our method. 

\subsection{Optimal sensor placement for Bayesian inverse problems}\label{sec3 1}
Here we focus on inverse problems where data is collected at a set of sensors. In this setting, the OED goal is finding an optimal set of locations for sensor deployment in some prescribed measurement domain $\measSpace$. To this end, we fix a candidate set of locations $\{\sens_i\}_{i=1}^{\nSens}$ (with each $\sens_i \in \measSpace$) and define the design problem as that of finding an optimal subset of locations from this set\footnote{This discrete approach to sensor placement is common in OED literature~\cite{alexanderian:oedreview}, however a continuous formulation is also possible~\cite{neitzel:oed}.}. To distinguish between different designs or sensor arrangements, a binary weight $w_i \in \{0,1\}$ is assigned to each sensor location $\sens_i$. If $w_i = 1$, then data is collected using the sensor at location $\sens_i$, whereas a weight of $0$ implies no measurement is conducted at $\sens_i$. For general Bayesian inverse problems, the optimal arrangement of $k$ sensors is then defined through the optimal weight vector $\weight^* \in \{0,1\}^{\nSens}$, which minimizes an expected utility function, \ie, 
\begin{equation}
    \weight^* = \argmin_{\weight \in \{0,1\}^{\nSens}} \mathbb{E}_{\data|\weight}\left[U(\weight,\data) \right] \quad \text{s.t.} \sum_{i=1}^{\nSens}w_i = \nChosen. 
    \label{eq:EUtility}
\end{equation}
The utility function $U$ assesses the effectiveness of using any sensor combination (defined via the weight vector $\weight \in \{0,1\}^{\nSens}$) in solving the Bayesian inverse problem with measurement data $\data \in \mathbb{R}^{\nData}$. Since the posterior measure depends on the data, minimizing the expectation of the utility function ensures that the chosen sensor placement works well on average for all possible realizations of the data. The utility function is typically problem-specific and some common choices for infinite-dimensional Bayesian OED are described in~\cite{alexanderian:oedreview}. For illustrative purposes, we focus on the \emph{A-optimality} criterion, though our approach could be extended to other utility functions. The $k$-sensor A-optimal design minimizes the expected value of the average posterior pointwise variance. Letting $\mathcal{C}_{\both\vert\datasub,w}(\weight,\data)$ denote the posterior covariance operator at a fixed design $\weight$ and measurement data $\data$, the A-optimal design can be obtained by solving the minimization problem~\eqref{eq:EUtility} with $U(\weight,\data) \coloneqq \trace(\mathcal{C}_{\both\vert\datasub,w}(\weight,\data))$.

For Bayesian inverse problems governed by nonlinear PTO maps or involving non-Gaussian prior measures, there is no closed-form expression for the data-dependent posterior covariance operator. Thus efficient techniques for approximating the expected utility are required. In the infinite-dimensional setting, the approximation is often carried out by combining a Gaussian approximation to the posterior measure with a sample average approximation (SAA) to the expectation, \eg,  using a Laplace or Gauss-Newton approximation (see~\cite{alexanderian:nonlinoed,long:fastnonlin}). 
As mentioned in~\Cref{sec2 2}, employing these techniques to approximate the SAA to the expected utility requires computing the MAP estimator for many data samples.  

\paragraph{The linear Gaussian case}
For linear Bayesian inverse problems with Gaussian priors and additive Gaussian noise, the A-optimality criterion simplifies, and an explicit formula is available. For notational convenience, we introduce a weight matrix $\Weight \in \mathbb{R}^{\nChosenObs \times \nData}$. In our target applications, we assume measurements can be obtained at each sensor for $\nTimes \in \mathbb{N}$ different times, so $\nChosenObs = \nChosen\nTimes$ (where $\nChosen \leq \nSens$ denotes the number of selected sensors) and $\nData = \nTimes\nSens$. The matrix $\Weight$ is thus a block-diagonal matrix with each diagonal block $\M{P}_{\weight} \in \{0,1\}^{k \times \nSens}$ defined as a sub-matrix of $\text{diag}(\weight)$ where the rows corresponding to zero weights have been removed (as described in~\cite{alexanderian:baeoed}). 

With this notation and the assumptions $\both \sim \mathcal{N}(\meanboth,\mathcal{C}_{\both\both})$ and $\noise \sim \mathcal{N}(\meannoise,\Cnoise)$, the design-dependent relationship between the data and parameters,  
\begin{equation}
    \data(\weight) = \Weight\left(\sfwd\both+\noise\right), 
    \nonumber
\end{equation}
induces a design-dependent Gaussian posterior $\mupostbothW{}$ with conditional mean and covariance 
\begin{align}\label{eq linGausW}
\meanpostbothW{}(\weight,\data) &=\CpostbothW{}\left(\sfwd^\ast\bvecS{\Sigma}(\weight)\left(\data-\meannoise\right)+\mathcal{C}_{\both\both}^{-1}\meanboth\right),\quad
\CpostbothW{}(\weight) =(\mathcal{C}_{\both\both}^{-1}+\sfwd^\ast\bvecS{\Sigma}(\weight)\sfwd)^{-1}, 
\end{align}
where $\bvecS{\Sigma}(\weight) \coloneqq \Weight^T\left(\Weight \Cnoise \Weight^T\right)^{-1}\Weight$. For linear Bayesian inverse problems, the posterior covariance operator is independent of the observed data, \ie, the expectation in~\eqref{eq:EUtility} is extraneous. Hence the A-optimal design can be found by minimizing the simplified criterion: 
\begin{equation}
    \weight^\ast = \argmin_{\weight \in \{0,1\}^{\nSens}} \trace\left[ \CpostbothW{}(\weight) \right], \quad \text{s.t.} \sum_{i=1}^{\nSens}w_i = k. 
    \label{eq:linearA-opt}
\end{equation}

\subsection{Invariance of the uncertainty-aware optimal design to linearization choice}\label{sec3 2}
In this section, we outline our approach to optimal sensor placement for Bayesian inverse problems where the accurate (but computationally costly) PTO map has been replaced with some linear surrogate. As in~\Cref{sec2 3}, with $\sfwd\in\mathcal{L}(\bothSpace,\mathbb{R}^{\nData})$ we denote the surrogate map that approximates the accurate nonlinear operator $\afwd:\bothSpace\rightarrow\mathbb{R}^{\nData}$. 

To incorporate model error into the design-dependent Bayesian inverse problem, we follow the BAE approach outlined in~\Cref{sec2 3}. The data measured at any subset of the candidate locations can be modeled by 
\begin{align*}
    \data(\weight) = \Weight\left(\afwd(\both)+\noise\right)= \Weight\left(\sfwd\both+\noise\right)+\varepsilon(\both,\weight) \approx \Weight\sfwd\both+\tot\vert\both(\weight),
\end{align*}
where $\varepsilon(\both,\weight) = \Weight\left(\afwd(\both)-\sfwd\both\right)$, and the random variable $\tot\vert\both(\weight)\sim\mathcal{N}(\Weight\meantotgboth,\Weight\Ctotgboth\Weight^T)$ is used to approximate the total error at the chosen sensor locations. The total error mean $\meantotgboth$ and covariance $\Ctotgboth$ are defined in~\eqref{eq:totErr_stats} and~\eqref{eq baeNorm} and can be estimated using Monte Carlo sampling as described in the end of~\Cref{sec2 3}. Since the design matrix enters linearly into the model, the design-dependent approximate posterior is also independent of the particular choice of linearization. That is, we can extend the results of~\Cref{thm:1} and~\Cref{cor:1} to design-dependent Bayesian inverse problems. This is done in the following corollary.

\begin{cor}\label{cor:2}
    Let $\bothSpace$ be a Hilbert space and assume $\both \in \bothSpace$ has prior measure $\mupr$ with mean $\meanboth$ and trace-class covariance operator $\Cbothboth$. Assume that
    $$
    \data(\weight) = \Weight\left(\afwd(\both)+\noise\right), 
    $$
    where $\afwd:\bothSpace\rightarrow\mathbb{R}^{\nData}$ is a bounded PTO map, the matrix $\Weight \coloneqq \Weight(\weight) \in \{0,1\}^{\nChosenObs \times \nData}$ is defined as described in~\Cref{sec3 1} for any $\weight \in \{0,1\}^{\nSens}$, and $\noise \sim \mathcal{N}(\meannoise,\Cnoise)$. 

    Then for any $\weight$, the approximate posterior parameterized by some linear map $\sfwd \in \mathcal{L}(\bothSpace,\mathbb{R}^{\nData})$, 
    $$
    \mupostbothW{\sfwd}(\weight,\data) = \mathcal{N}(\meanpostbothW{\sfwd}(\weight,\data),\CpostbothW{\sfwd}(\weight)), 
    $$
    where 
    \begin{align*}
    \CpostbothW{\sfwd}(\weight) &= \left(\Cbothboth^{-1}+\widetilde{\sfwd}^{\ast}\Weight^T(\Weight\Ctotgboth\Weight^T)^{-1}\Weight\widetilde{\sfwd}\right)^{-1},\\
    \meanpostbothW{\sfwd}(\weight,\data) &= \CpostbothW{\sfwd}(\weight)\left(\widetilde{\sfwd}^{\ast}\Weight^T(\Weight\Ctotgboth\Weight^T)^{-1}\Weight\widetilde{\data}+\Cbothboth^{-1}\meanboth\right),\\
    \widetilde{\sfwd} &= \sfwd+\Cepsboth\Cbothboth^{-1}, \quad \widetilde{\data} = \data-\meantotgboth = \data-\meannoise-\meaneps+\Cepsboth\Cbothboth^{-1}\meanboth, \\
    \Ctotgboth &= \Cnoise+\Cepseps-\Cepsboth\Cbothboth^{-1}\Cbotheps, 
    \end{align*}
    is independent of the choice of $\sfwd$. In particular, for any $\sfwd_1,\sfwd_2 \in \mathcal{L}(\bothSpace,\mathbb{R}^{\nData})$, we have:
    $$
    \trace\left[\CpostbothW{\sfwd_1}(\weight)\right] = \trace\left[\CpostbothW{\sfwd_2}(\weight)\right].
    $$
\end{cor}
\begin{proof}
   Note that as a consequence of~\Cref{thm:1}, $\widetilde{\sfwd}$, $\widetilde{\data}$ and $\Ctotgboth$ are independent of the choice of $\sfwd$, and the result thus follows. 
\end{proof}

An immediate consequence of the invariance of the A-optimality criterion to the specific linearization is that the uncertainty-aware A-optimal designs are also independent of the specific choice of surrogate map $\sfwd$. Thus the trivial map, the {\em zero operator} $\mathsf{O}\colon \both \mapsto \bvecS{0}$, leads to the same optimal sensor placements as a more complicated tailored surrogate when model error is incorporated using the BAE approach. This is the key insight that facilitates our greedy, derivative-free approach to optimal experimental design, which we describe in~\Cref{subsec:discreteoed}.   

\begin{rmk}
As mentioned previously, in practice, the results in~\Cref{thm:1} and thus~\Cref{cor:2} only hold asymptotically as the number of samples used in the MC approximation of the model error statistics go to infinity. In particular, in~\cite{nicholson:linearBAE}, it was shown that the use of the zero map led to posteriors with underestimated pointwise variance when a small number of samples (typically $< 1000$) was used. Thus, if the cost of evaluating $\afwd$ significantly limits $N$, it is advisable to use a surrogate $\sfwd$ that is highly correlated with $\afwd$. 
\end{rmk}

\paragraph{Uncertainty-aware marginalized OED}\label{sec3 3}

As alluded to in~\Cref{sec2 3}, in many inverse problems there may be uncertain auxiliary parameters which are not estimated. In such settings, designs are chosen to minimize uncertainty in the marginal posterior for the primary parameters-of-interest.
It is straightforward to extend our uncertainty-aware approach to approximate optimal sensor placements for such marginalized design problems.
Decompose $\both = (\m,\aux)$ into a primary parameter-of-interest $\m \in \mSpace$ and an auxiliary parameter $\aux \in \auxSpace$, and let $\Pi_{\mSpace}:\bothSpace\to\mSpace$ denote linear a projection operator onto $\mSpace$.  
The approximate marginal posterior measure, parameterized by some linear surrogate $\sfwd$, is $
\mupostmW{\sfwd} \sim \mathcal{N}(\Pi_{\mSpace}\meanpostbothW{\sfwd},\Pi_{\mSpace}\CpostbothW{\sfwd}\Pi_{\mSpace}^{\ast}) \eqqcolon \mathcal{N}(\m_{0\vert\datasub,w}^{\sfwd},\CpostmW{\sfwd}),
$
and the optimal sensor placements for the marginalized design problem thus satisfy
\begin{equation}
    \weight^* = \argmin_{\weight \in \{0,1\}^s} \trace\left[\CpostmW{\sfwd}(\weight)\right], \quad \text{s.t.} \sum_{i=1}^{\nSens} w_i = \nChosen.
    \label{eq:margOED}
\end{equation}
Since $\meanpostbothW{\sfwd}$ and $\CpostbothW{\sfwd}$ are independent of the choice of $\sfwd$ (using~\Cref{cor:2}), as is the projection operator $\Pi_{\mSpace}$, it is straightforward to see that the (approximate) optimal design for the marginalized problem is independent of the linear surrogate $\sfwd$. 
As a consequence, the zero operator is also a valid surrogate for the marginalized OED problem. 

\subsection{Efficient computation of the greedy A-optimal designs}\label{subsec:discreteoed}

In this section, we outline a greedy procedure for computing A-optimal sensor placements for both design problems~\eqref{eq:linearA-opt} and~\eqref{eq:margOED}. Although we emphasize that our approach could be used with the zero map surrogate (and this is what we will use for the numerical examples in~\Cref{sec4} and~\Cref{sec5}), for generality we present the procedure using an arbitrary linear surrogate map $\sfwd$. 

For the remainder of this section, let $\bothdisc \in \mathbb{R}^{\nboth}$, $\afwddisc \colon \mathbb{R}^{\nboth} \rightarrow \mathbb{R}^{\nData}$, and $\sfwddisc \in \mathcal{L}(\mathbb{R}^{\nboth},\mathbb{R}^{\nData})$ denote the discretized parameter $\both$, accurate PTO map $\afwd$, and linear surrogate PTO map $\sfwd$, respectively. The discretized uncertainty-aware model is thus
\begin{align*}
    \data(\weight) = \Weight(\sfwddisc\bothdisc + \tot\vert\bothdisc), 
\end{align*}
where $\tot\vert\bothdisc \sim \mathcal{N}(\meantotgboth,\Ctotgboth)$, with $\meantotgboth = \meannoise+\meaneps-\M{C}_{\epssub\both}\M{C}_{\both\both}^{-1}\left(\bothdisc-\bothdisc_0\right)$ and $\Ctotgboth = \Cnoise+\Cepseps-\M{C}_{\epssub\both}\M{C}_{\both\both}^{-1}\M{C}_{\both\epssub}$. Here, the discretized statistics $\bothdisc_0,\, \meaneps,\, \M{C}_{\both\both},\, \M{C}_{\both\epssub} = \M{C}^{\ast}_{\epssub\both}$, and $\Cepseps$ are computed via Monte Carlo (as in~\eqref{eq epsStats1}) using the discretized operators $\sfwddisc$ and $\afwddisc$. Letting $\widetilde{\sfwddisc} = \sfwddisc+\M{C}_{\epssub\both}\M{C}_{\both\both}^{-1}$, the corresponding discretized posterior covariance operator is then 
\begin{align}
    \CpostbothWdisc{\sfwddisc}(\weight) &= \left(\M{C}_{\both\both}^{-1} + \widetilde{\sfwddisc}^{\ast}\Weight^T\left(\Weight\Ctotgboth\Weight^T\right)^{-1}\Weight\widetilde{\sfwddisc}\right)^{-1} \\
    &= \M{C}_{\both\both} - \left(\M{C}_{\both\both}{\sfwddisc}^{\ast}+\M{C}_{\both\eps}\right)\Weight^T\M{Q}(\Weight)\Weight\left(\sfwddisc\M{C}_{\both\both}+\M{C}_{\eps\both}\right)
    \label{eq:postSMW}
\end{align}
where $\widetilde{\sfwddisc}^{\ast}$ denotes the adjoint of $\widetilde{\sfwddisc}$, $\M{Q}(\Weight) \coloneqq \left[\Weight\left(\Ctotgboth+\widetilde{\sfwddisc}\M{C}_{\both\both}\widetilde{\sfwddisc}^{\ast}\right)\Weight^T\right]^{-1}$ and the last equation follows from the Woodbury matrix identity and $\widetilde{\sfwddisc}\M{C}_{\both\both} = \sfwddisc\M{C}_{\both\both}+\M{C}_{\epssub\both}$. 

Using the last equality and the cyclic property of the trace operator, we have
\begin{align*}
\trace \left[ \CpostbothWdisc{\sfwddisc}(\weight) \right] = \trace\left[\M{C}_{\both\both}\right] - \trace\left[\M{Q}(\Weight)\Weight \left(\sfwddisc\M{C}_{\both\both}+\M{C}_{\epssub\both}\right) \left(\M{C}_{\both\both}{\sfwddisc}^{\ast}+\M{C}_{\both\epssub}\right)\Weight^T \right].
\end{align*}
Thus, letting $\M{K}(\weight) \coloneqq \M{Q}(\Weight)\Weight \left(\sfwddisc\M{C}_{\both\both}+\M{C}_{\epssub\both}\right) \left(\M{C}_{\both\both}{\sfwddisc}^{\ast}+\M{C}_{\both\epssub}\right)\Weight^T $ the A-optimal design satisfies
\begin{equation}
    \weight^* = \argmax_{\weight \in \{0,1\}^{\nSens}} \trace\left[\M{K}(\weight)\right] \quad \text{s.t.} \sum_{i=1}^{\nSens} w_i = \nChosen.
    \label{eq:discAopt}
\end{equation}
This reformulation of the A-optimality criterion is computationally advantageous if $\nboth \gg \nData$, which is often the case for inverse problems, particularly when dealing with discretizations of infinite-dimensional parameters $\both$. 

\paragraph{Discrete marginalized A-optimal design problem} 
In the case where $\both = (\m,\aux)$, and we are interested in finding designs that optimize the marginalized A-optimality criterion~\eqref{eq:margOED}, we again let $\bothdisc = [\mdisc,\auxdisc]$, with $\mdisc \in \mathbb{R}^{\nm}$ and $\auxdisc \in \mathbb{R}^{\naux}$ (with $\nboth = \nm+\naux$) denoting the discretized primary and auxiliary parameters, respectively. Additionally, we decompose ${\sfwddisc} = [\M{F}_{\m} \quad \M{F}_{\aux}]$, $\M{C}_{\epssub\both} = [\M{C}_{\epssub\m}\quad\M{C}_{\epssub\aux}]$, and $\M{C}_{\both\both} = \begin{bmatrix}
 \M{C}_{\m\m}
 &  \M{C}_{\m\aux} \\
\M{C}_{\aux\m}
 &
 \M{C}_{\aux\aux}
\end{bmatrix}$. Using~\eqref{eq:postSMW}, it is straightforward to see that 
\begin{equation}
    \M{C}_{\m\vert\datasub,w}(\weight) = \M{C}_{\m\m}-\left(\M{C}_{\m\m}\M{F}^{\ast}_{\m}+\M{C}_{\m\aux}\M{F}^{\ast}_{\aux}+\M{C}_{\m\epssub}\right)\Weight^T\M{Q}(\Weight)\Weight\left(\M{F}_{\m}\M{C}_{\m\m}+\M{F}_{\aux}\M{C}_{\aux\m}+\M{C}_{\epssub\m}\right),\nonumber.
\end{equation}
Denoting $\M{K}_{\m}(\weight) \coloneqq \M{Q}(\Weight)\Weight(\M{F}_{\m}\M{C}_{\m\m}+\M{F}_{\aux}\M{C}_{\aux\m}+\M{C}_{\epssub\m})(\M{C}_{\m\m}\M{F}^{\ast}_{\m}+\M{C}_{\m\aux}\M{F}^{\ast}_{\aux}+\M{C}_{\m\epssub})\Weight^T$, 
the marginal A-optimal design thus satisfies
\begin{equation}
\weight^* = \argmax_{\weight \in \{0,1\}^{\nSens}} \trace\left[\M{K}_{\m}(\weight)\right] \quad \text{s.t.} \sum_{i=1}^{\nSens} w_i = \nChosen.
    \label{eq:discMargAopt}
\end{equation}

Note that~\eqref{eq:discAopt} and~\eqref{eq:discMargAopt} define NP-hard optimization problems. Approaches to solving this challenging problem often involve relaxation techniques~\cite{alexanderian:oed,herman:reweightl1,yu:sumup}, although relaxation-free approaches are also available (see, \eg,~\cite{attia:stochasticOED}). However, here we take a greedy approach to finding the placement of $k$ sensors that approximately solves the optimization problem~\eqref{eq:discAopt}. In the greedy approach to sensor placement, sensors that lead to the largest decrease in the A-optimality criterion are chosen one at a time from the (diminishing) candidate set. The procedure is outlined in~\cref{alg:oed}. The greedy approach to sensor placement generally yields sub-optimal designs. However, it has been shown to produce reasonable designs in practice~\cite{alexanderian:marginal,alexanderian:baeoed,jagalur:batch}. Additionally, the greedy approach to sensor placement is completely hands-off (since no derivative information of the objective is required). Thus, when combined with the zero map surrogate, $\sfwd \equiv \zero$, our proposed approach is a black box that only requires sample parameter and data pairs.

At each step $\ell = 1,\ldots,\nChosen$, the objective function defined in~\eqref{eq:discAopt} (resp.~\eqref{eq:discMargAopt} in the marginal OED case) needs to be evaluated $\nSens-(\ell-1)$ times. Once the matrices $\Ctotgboth+\widetilde{\M{F}}\M{C}_{\both\both}\widetilde{\M{F}}^{\ast}$ and \linebreak $\left(\sfwddisc\M{C}_{\both\both}+\M{C}_{\epssub\both}\right) \left(\M{C}_{\both\both}{\sfwddisc}^{\ast}+\M{C}_{\both\epssub}\right)$ (resp. $(\M{F}_{\m}\M{C}_{\m\m}+\M{F}_{\aux}\M{C}_{\aux\m}+\M{C}_{\epssub\m})(\M{C}_{\m\m}\M{F}^{\ast}_{\m}+\M{C}_{\m\aux}\M{F}^{\ast}_{\aux}+\M{C}_{\m\epssub})$) are precomputed, each evaluation of the objective only requires solving a system with a matrix of dimension $\ell\nTimes$ (where $\nTimes$ denotes the number of times measurements are collected at each sensor) and summing up the diagonal entries. The computational bottleneck for most problems with PDE-dependent PTO maps will be the computation of the statistics (\eqref{eq epsStats1}-\eqref{eq epsStats3}) required for modeling the error using the BAE approach, since this requires solving $q$ PDEs with the accurate (costly) PTO map. We note that for our proposed zero map surrogate, precomputing the matrices in question is straightforward. In particular, $\M{Q} = \left[\M{W}\left( \Cnoise + \Cepseps \right)\M{W}^T\right]^{-1}$ and the equations for $\M{K}(\weight)$ and $\M{K}_{m}(\weight)$ simplify:  
\begin{align*}
\M{K}(\weight) &= \left[\M{W}\left( \Cnoise + \Cepseps \right)\M{W}^T\right]^{-1}\Weight\M{C}_{\epssub \both}\M{C}_{\both \epssub}\Weight^T \\
\M{K}_{m}(\weight) &= \left[\M{W}\left( \Cnoise + \Cepseps \right)\M{W}^T\right]^{-1}\Weight\M{C}_{\epssub m}\M{C}_{m \epssub}\Weight^T. 
\end{align*}
However, if $\M{F}$ is defined via a discretized (linear) PDE, then some approximation to $\M{F}$ may be required. There are various ways of constructing such approximations efficiently, \eg, using randomized matrix algorithms~\cite{halko:randsvd}. 
\begin{algorithm}[ht]
\caption{Algorithm for finding A-optimal uncertainty-aware greedy sensor placements}\label{alg:oed}
\hspace*{\algorithmicindent} \textbf{Input} Target number of sensors $k$, prior distribution $\mupr$, mean $\meannoise$ and covariance $\Cnoise$ of the observational noise, accurate and (linear) surrogate parameter-to-observable maps $\afwddisc$, $\sfwddisc$ \\
\hspace*{\algorithmicindent} \textbf{Output} Optimal weight vector $\weight^*$
\begin{algorithmic}[1]
\STATE Create the tuple $\{(\bothdisc^{(\ell)}, \eps^{(\ell)})\}_{\ell=1}^{N}$ of parameter samples $\bothdisc^{(\ell)} \sim \mupr $ and corresponding approximation error samples $\eps^{(\ell)}$ using~\Cref{eq:approxErr} 
\STATE Compute sample means $\meaneps, \bothdisc_0$ and (cross-)covariances $\Cepseps, \M{C}_{\eps\both},\M{C}_{\both\eps}, \M{C}_{\both\both}$ using~\eqref{eq epsStats1}-\eqref{eq epsStats3}
\STATE Precompute $\Ctotgboth+\widetilde{\M{F}}\M{C}_{\both\both}\widetilde{\M{F}}^{\ast}$ and $\left(\sfwddisc\M{C}_{\both\both}+\M{C}_{\epssub\both}\right) \left(\M{C}_{\both\both}{\sfwddisc}^{\ast}+\M{C}_{\both\epssub}\right)$ \\ \hspace{20mm}\{or $(\M{F}_{\m}\M{C}_{\m\m}+\M{F}_{\aux}\M{C}_{\aux\m}+\M{C}_{\epssub\m})(\M{C}_{\m\m}\M{F}^{\ast}_{\m}+\M{C}_{\m\aux}\M{F}^{\ast}_{\aux}+\M{C}_{\m\epssub})$ in marginal case\}
\STATE Initialize candidate and chosen index sets: $\mathcal{I}_{\text{chosen}} \gets \varnothing$, $\mathcal{I}_{\text{candidate}} \gets \{1,\ldots,\nSens\}$
\FOR{$i = 1\cdots k$}
\STATE $\weight^*(\mathcal{I}_{\text{chosen}}) = 1$
\STATE $I_i = \argmax_{j \in \mathcal{I}_{\text{candidate}}} \trace\big[\M{K}(\weight^*+\M{e}_j)\big]$ \COMMENT{$\M{e}_j \in \mathbb{R}^{\nSens}$ denotes the $j$-th coordinate vector} 
\\ \hspace{30mm} \{or $\trace\big[\M{K}_{\m}(\weight^*+\M{e}_j)\big]$ in marginal case\}
\STATE $\mathcal{I}_{\text{chosen}} \gets \mathcal{I}_{\text{chosen}} \cup I_i$
\STATE $\mathcal{I}_{\text{candidate}} \gets \mathcal{I}_{\text{candidate}} \setminus \mathcal{I}_{\text{chosen}}$
\ENDFOR
\RETURN{$\weight^*$}
\end{algorithmic}
\end{algorithm}
In the next two sections we illustrate the application of our proposed linearize-then-optimize approach for computing optimal designs for two inverse problems. Specifically, for an idealized subsurface flow problem and for a tsunami detection problem.

\section{Numerical example 1: Idealized subsurface flow}\label{sec4}
The first numerical example we consider is an idealized subsurface flow problem. Specifically, we consider the inverse problem of estimating the the (log-)boundary flux $\m$ and the (log-)permeability field $\aux$ in a two dimensional subsurface aquifer from noisy (point-wise) measurements of the (head) pressure $u$. The boundary flux and permeability field are related to the pressure by the steady state Darcy flow equation within the aquifer:
\begin{equation}
\begin{aligned}
    \label{eq poisFwd}
\nabla\cdot(\exp{(\aux)} \nabla u)&=0\quad &&\text{in }\Omega,\\
    (\exp{(\aux)}\nabla u)\cdot \bs{n}&=\exp{(\m)}\quad &&\text{on }\Gamma_{\rm T},\\
    (\exp{(\aux)}\nabla u)\cdot \bs{n}&=-1\quad &&\text{on }\Gamma_{\rm B},\\
    u&=0\quad &&\text{on }\Gamma_{\rm L},\\
    (\exp{(\aux)}\nabla u)\cdot \bs{n}&=0\quad &&\text{on }\Gamma_{\rm R},
    \end{aligned}
\end{equation}
with $\bs{n}$ denoting the outward normal vector. For a more in-depth discussion on the physical interpretation of the problem, we refer to ~\cite{carrera1986flow1,carrera1986flow2,carrera1986flow3}. Similar problems are commonly considered as benchmark tests~\cite{kovachki2023neural,iglesias2013ensemble,alexanderian:nonlinoed}. 

In the current work we employ a level set approach\cite{dunlop2017hierarchical,iglesias2016bayesian} for the parametrization of the permeability $\aux$. Specifically, we take 
\begin{align}\label{eq: levset}
\aux=\Phi(\psi)=\sum_{i=1}^L\mathsf{z}_i \chi_{D_i(\psi)},  
\end{align}
where $\chi$ is the indicator function, $D_i(\psi)=\{x\in \Omega \, | \, \ell_{i-1}\leq \psi(x) \leq \ell_i \}$, and $\mathsf{z}_i\in\mathbb{R}$ for $i=1,2,\dots,L$ are the discrete ($\log$-)permeability values. Thus $\Phi$ maps the so-called level set field $\psi$ to the permeability $\aux$. We make explicit that although the permeability is the parameter of interest, inference and the associated optimal experimental design  problem are carried out for the the level set field $\psi$.

\subsection{Problem set up}
We consider the problem in the unit square; $\Omega=(0,1)\times(0,1)$, and set $\Gamma_{\rm T}$, $\Gamma_{\rm B}$, $\Gamma_{\rm L}$ and $\Gamma_{\rm R}$ as the top, bottom and left and right boundaries, respectively, \ie,  $\Gamma_{\rm T}:=(0,1)\times\{1\}$, $\Gamma_{\rm B}:=(0,1)\times\{0\}$,  $\Gamma_{\rm L}:=\{0\}\times (0,1)$, and $\Gamma_{\rm R}:=\{1\}\times (0,1)$. A Galerkin finite element method is used to solve the forward problem (\ref{eq poisFwd}). More specifically, the domain is discretized into 3042 triangular elements and   1600 continuous piece-wise linear basis functions. Thus the pressure $u$ as well as the permeability have dimension 1600. On the other hand, the unknown log-boundary flux is 40 dimensions. We impose a single level set threshold at $\ell=0$ on the level set field, with the permeability values (see (\ref{eq: levset})) $\mathsf{z}_1=0$ when $\psi\leq 0$, and $\mathsf{z}_2=1$ when $\psi>0$. Finally, we take a total of 1600  candidate sensor locations which are distributed in an equally spaced 40 by 40 grid-like array (see Figure \ref{fig:DarcyLocs1}). In Figure~\ref{fig:darcy1} we show the true underlying level set field $\psi_{\rm true}$, the resulting true permeability $\aux_{\rm true}=\Phi(\psi_{\rm true})$, the true boundary flux $\m_{\rm true}$, and the resulting true pressure $u_{\rm true}$.
\begin{figure}[t!]
\centering
\begin{tikzpicture}
\node[inner sep=0pt] (a) at (0,0)
{\includegraphics[height=.2\textwidth]{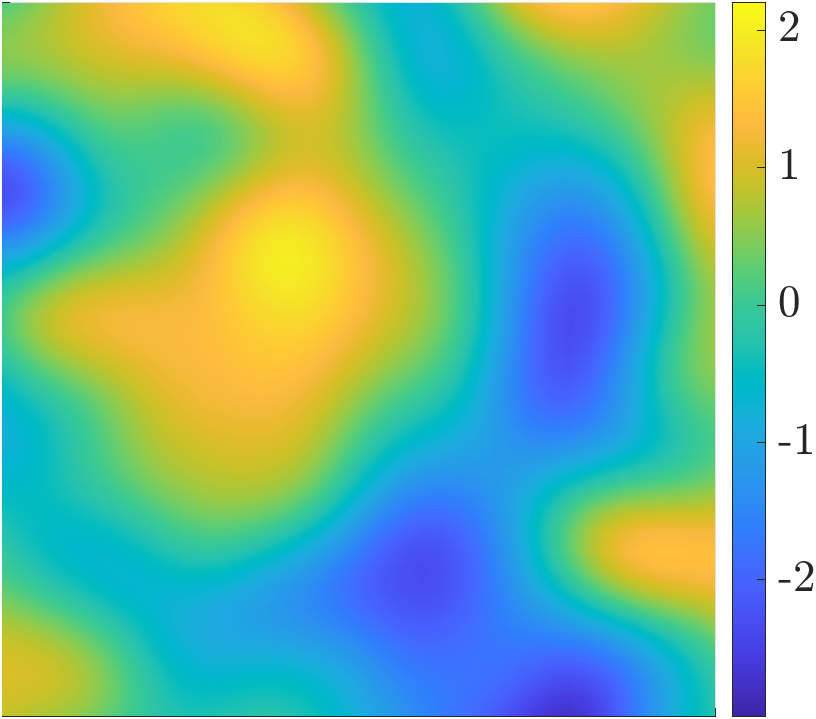}};
\node[inner sep=0pt] (a) at (3.9,0)
{\includegraphics[height=.2\textwidth]{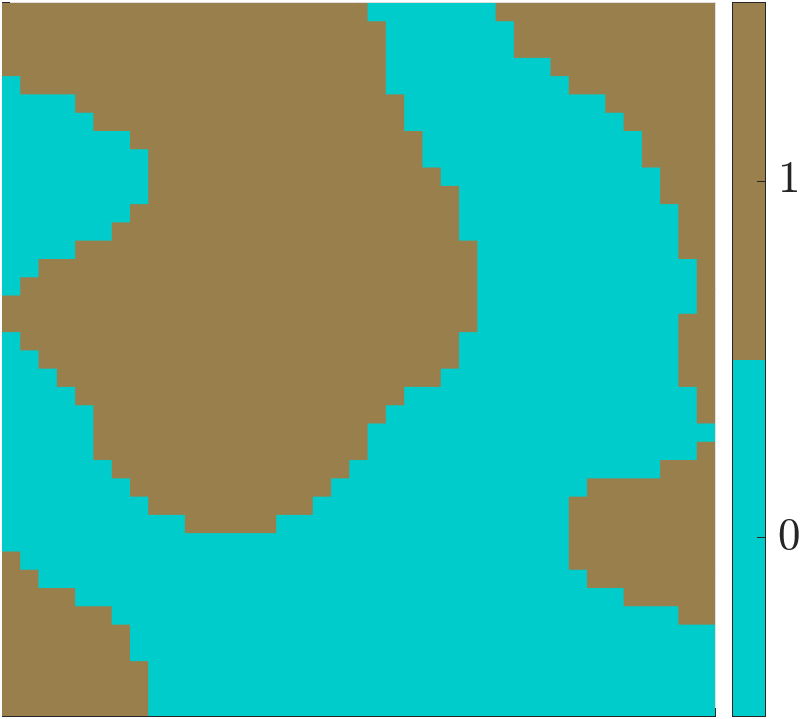}};
\node[inner sep=0pt] (a) at (7.8,0)
{\includegraphics[height=.2\textwidth]{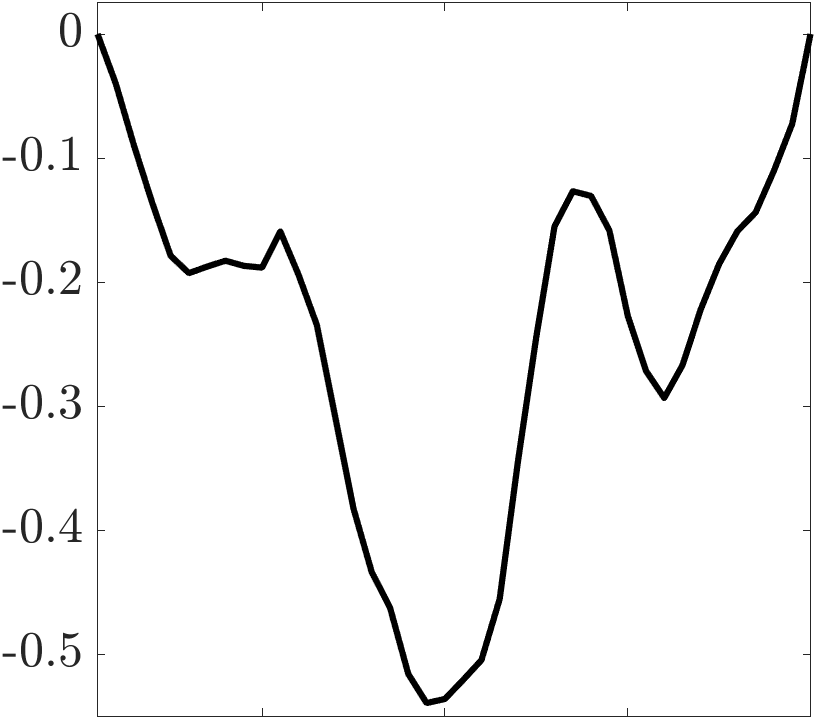}};
\node[inner sep=0pt] (a) at (11.8,0)
{\includegraphics[height=.2\textwidth]{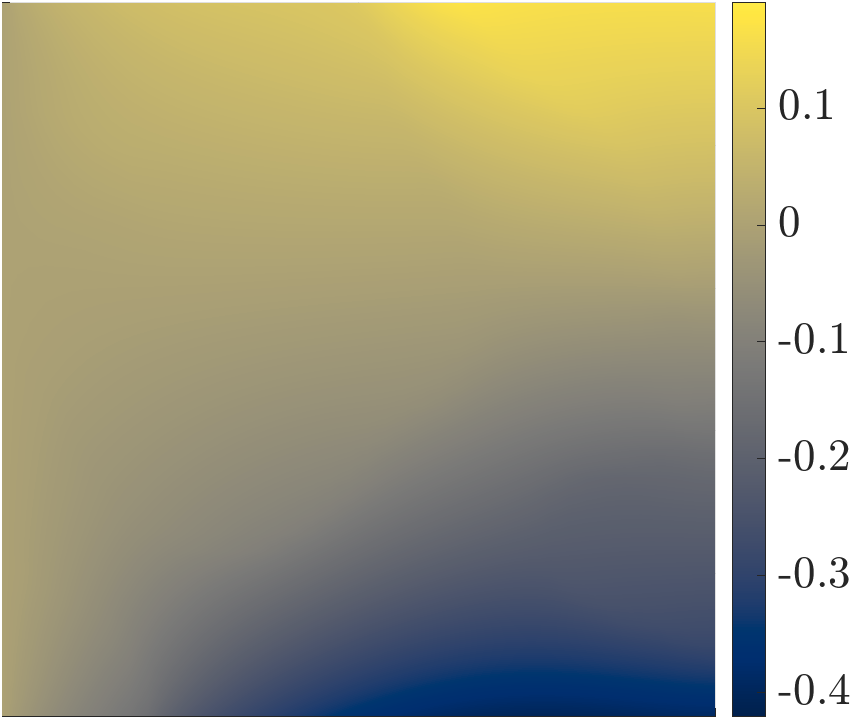}};
\end{tikzpicture}
\caption{The true underlying level set field $\psi_{\rm true}$ (far left), the induced true ($\log$-) permeability $\aux_{\rm true}=\Phi(\psi_{\rm true})$ (center left), the true ($\log$-) boundary flux $\m_{\rm true}$ through $\Gamma_{\rm T}$ (center right), and the resulting pressure $u_{\rm true}$ (far right).}
\label{fig:darcy1}
\end{figure}

\paragraph{Prior models} 
We assume the parameters $\m$ and $\aux$ are a priori independent, i.e., $\mu_{\m,\aux}=\mu_\m\mu_\aux$.  Furthermore, the (marginal) prior measure for the permeability is the push-forward of the prior measure of the level set field \cite{dunlop2017hierarchical,dunlop2021stability,iglesias2016bayesian}, \ie,  $\mu_\aux=\Phi_{\sharp} \mu_\psi$. As such, the joint prior used here is of the form $\mu_{\m,\psi}=\mu_m\mu_\psi$.

We postulate Gaussian prior measures for both the top boundary flux and the underlying level set field of the permeability, \ie, 
$\mu_\m=\mathcal{N}(\m_0,\mathcal{C}_{\msub\msub})$ and $\mu_\psi=\mathcal{N}(\psi_0,\mathcal{C}_{\psi\psi})$. The covariance operators are defined by
\begin{align}
\mathcal{C}_{\msub\msub}&=\mathcal{A}^{-2},\quad \text{with }\mathcal{A}=c_2-c_3\frac{d^2}{dx^2},\\
\mathcal{C}_{\psi\psi}v(x)&=\int_{\Omega}c(x,y)v(y)\,dy\quad\text{with } c(x,y)=\exp{\left(-\frac{\norm[\mathbb{R}^2]{x-y}^2}{c_1^2}\right)},\label{eq: covW}
\end{align}
with $v(x)\in L^2(\Omega)$ and the domain of $\mathcal{A}$ being $\mathcal{D}(\mathcal{A}):=H^2_0(\Gamma_{\rm T})$. The parameter $c_1$ controls the characteristic length scale of samples of $\psi$, while $c_2$ and $c_3$ control correlation length and variance of the samples of $\m$. Such covariance operators are fairly standard for subsurface flow Bayesian inverse problems~\cite{iglesias2016bayesian,holbach2023bayesian,alexanderian:nonlinoed}. In the current example we set $c_1=\frac{1}{8}$, $c_2=2$, and $c_3=8\times 10^{-2}$. In Figure~\ref{fig:darcy2} we show several prior samples $\m$, samples pf $\psi$ along with the corresponding samples of the permeability $\aux$.
\begin{figure}[t!]
\centering
\begin{tikzpicture}
\node[inner sep=0pt] (a) at (0,0)
{\includegraphics[height=.2\textwidth]{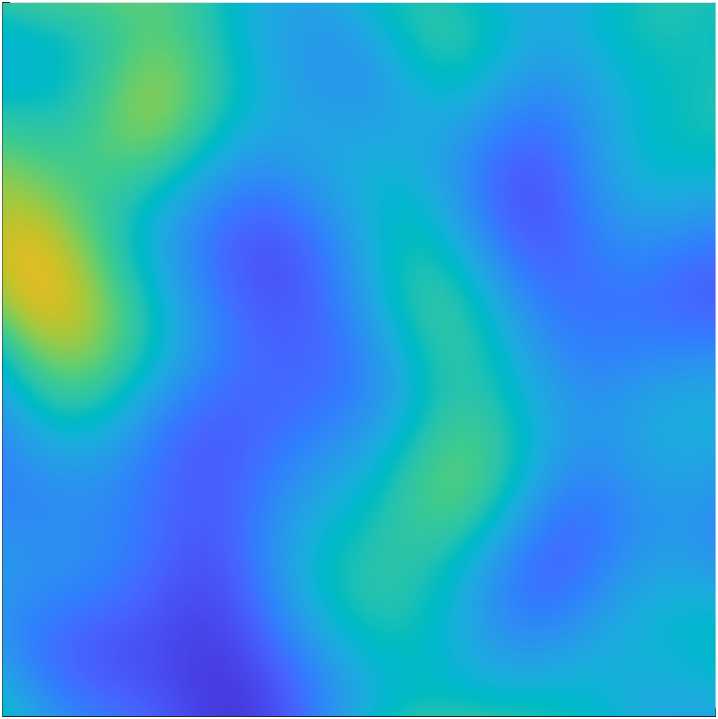}};
\node[inner sep=0pt] (a) at (4,0)
{\includegraphics[height=.2\textwidth]{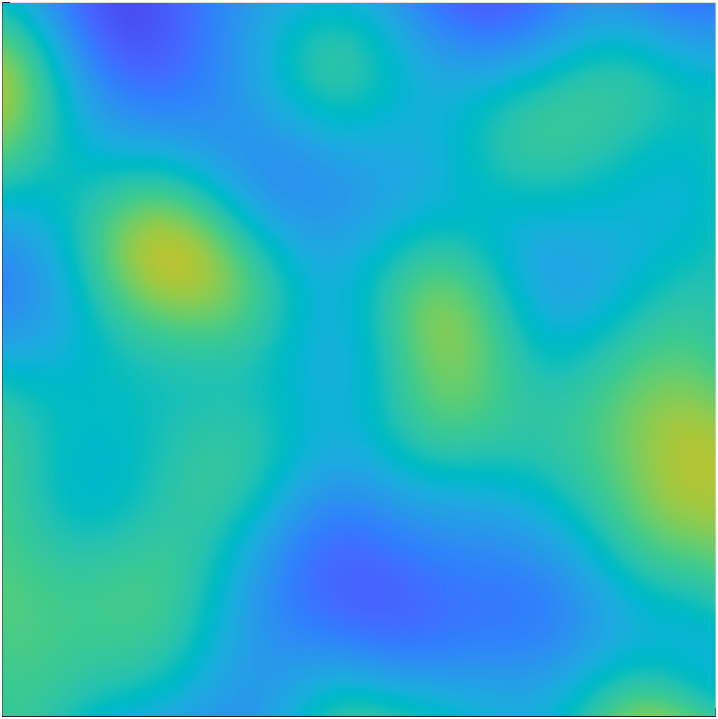}};
\node[inner sep=0pt] (a) at (8,0)
{\includegraphics[height=.2\textwidth]{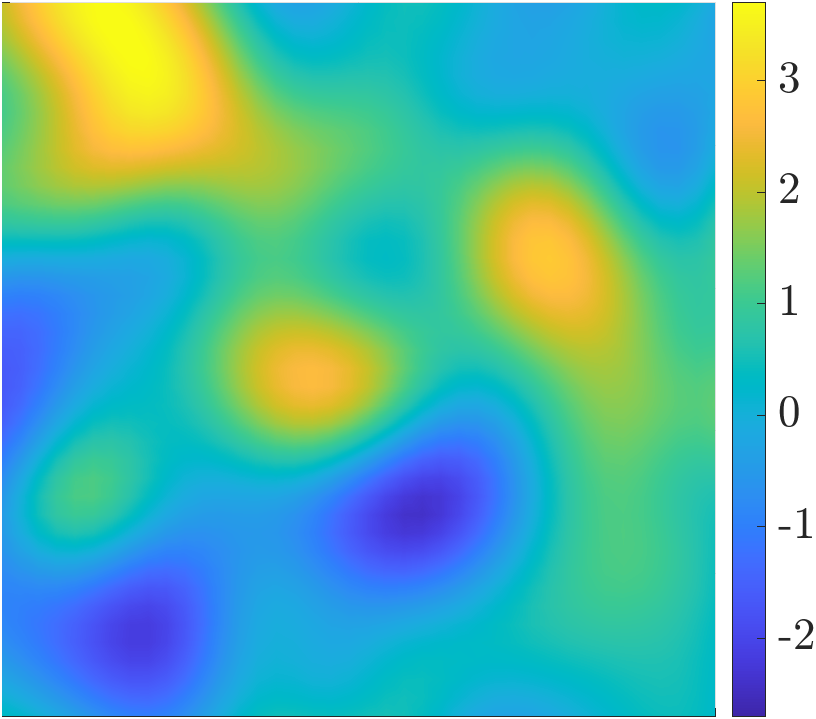}};
\node[inner sep=0pt] (a) at (12,-2)
{\includegraphics[height=.2\textwidth]{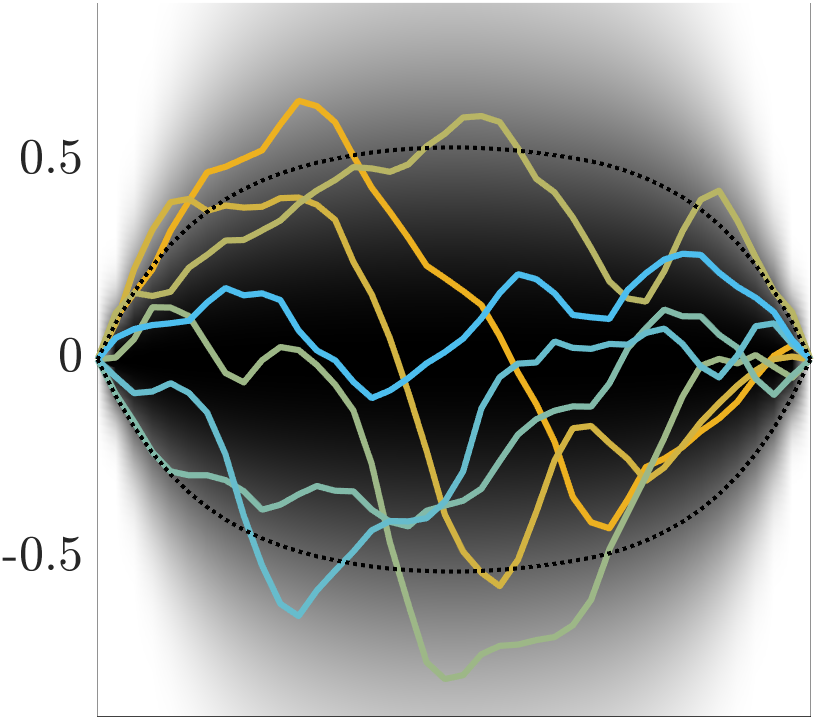}};
\node[inner sep=0pt] (a) at (0,-4)
{\includegraphics[height=.2\textwidth]{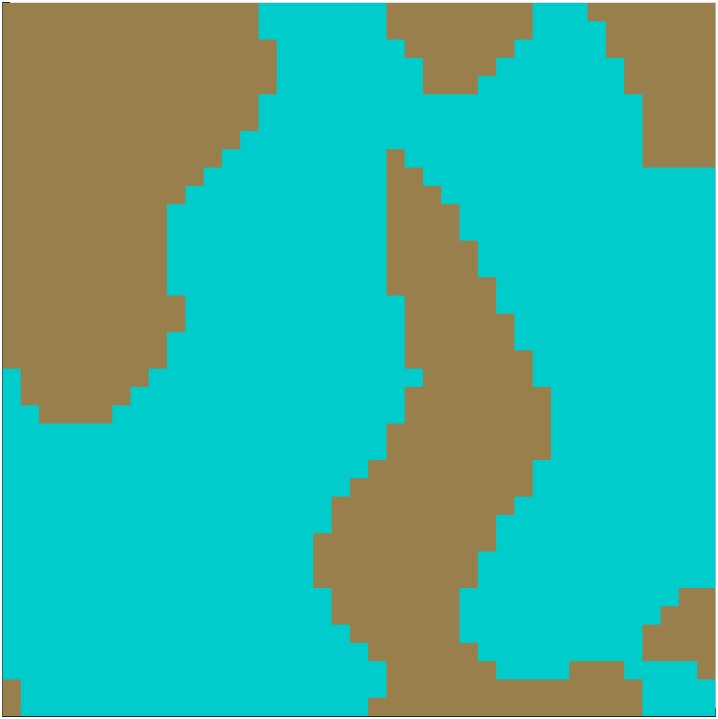}};
\node[inner sep=0pt] (a) at (4,-4)
{\includegraphics[height=.2\textwidth]{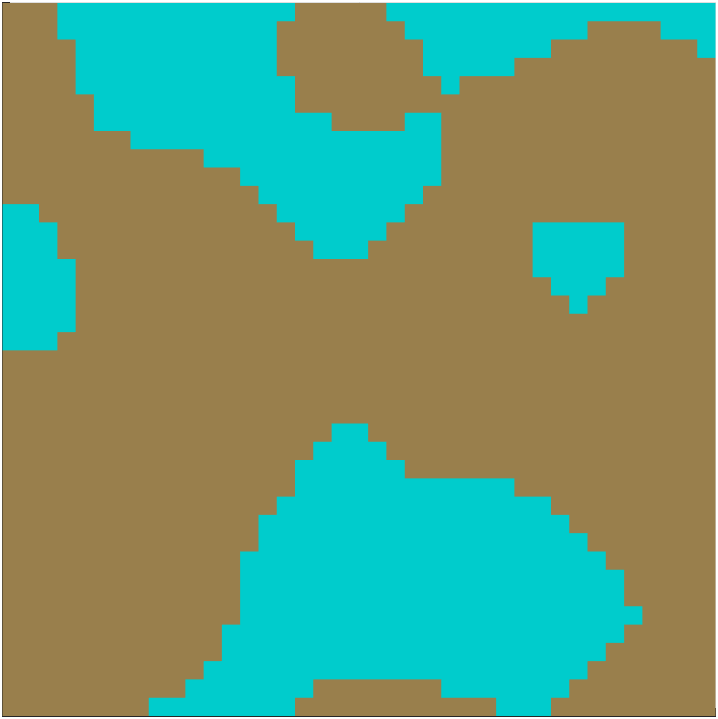}};
\node[inner sep=0pt] (a) at (8,-4)
{\includegraphics[height=.2\textwidth]{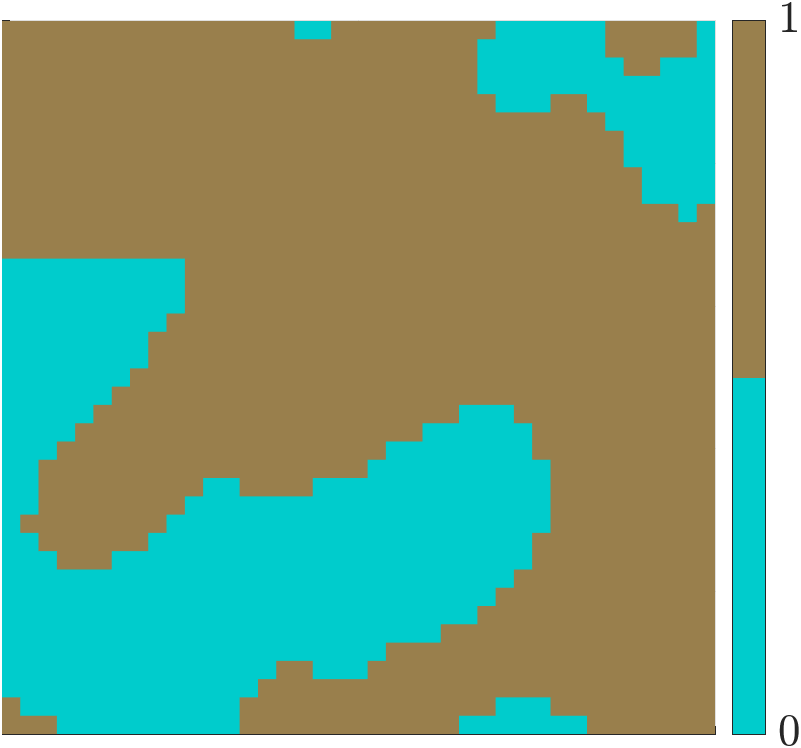}};
\end{tikzpicture}
\caption{Samples from the prior. Top row are samples of $\psi$, bottom row are corresponding samples of $\aux$. Samples of $\m$ are visualized in the right-most column where the grey-scale represents the prior density and the black dotted lines the plus/minus (marginal) one standard deviation.}
\label{fig:darcy2}
\end{figure}

\paragraph{Optimal experimental design} We consider several cases of OED for this numerical example, each carried out using the proposed linearize-then-optimize procedure. In particular, we consider joint OED (of both $\m$ and $\aux$), and then marginal OED for $\m$. For the marginal OED case we also compare the results found when neglecting the uncertainty in $\aux$. Thus, for the idealized subsurface flow example, we consider the following three OED problems:
\begin{enumerate}
    \item Joint OED for $\m$ and $\aux$,
    \item Marginal OED for $\m$ while accounting for the uncertainty in $\aux$, and
    \item Marginal OED for $\m$
    while ignoring the uncertainty in $\aux$.
\end{enumerate}

\paragraph{Linearization} In all cases we use the zero operator $\mathsf{O}\in\mathcal{L}(\mathcal{D}(\mathcal{A})\times L^2(\Omega),\mathbb{R}^{\nData})$ defined by $\mathsf{O}v=0$ for all $v\in \mathcal{D}(\mathcal{A})\times L^2(\Omega)$ as the approximate linear forward model. The zero model is particularly straight-forward to implement, completely non-invasive (\ie, well-suited in the case of {\em black-box} models), computationally cheap, and avoids the differentiability issues associated with the level set parameterization. More specifically, we generate $q=10,000$ samples from the joint prior, \ie, $(\m^{(\ell)},\psi^{(\ell)})\sim\mu_{\m,\psi}$ for $\ell=1,2,\dots,10,000$, and compute the corresponding (noiseless) model predictions, $\data^{(\ell)}$. Subsequently, all required quantities are computed using Equations (\ref{eq epsStats1})-(\ref{eq epsStats3}).

\subsection{Reference case}
Before considering the OED problem we solve the Bayesian inverse problem based on measurements collected at all candidate sensor locations. Specifically, we compute the approximate joint posterior for $\m$ and $\aux$ based on the linear(-ized) BAE approach outlined in Section \ref{sec2 3} and the ``true'' joint posterior with the full nonlinear PTO map using Markov chain Monte Carlo (MCMC) with the preconditioned Crank–Nicolson (pCN) algorithm~\cite{chen2018dimension,cotter2013mcmc}. This reference task serves to provide a {\em best case} scenario in terms of uncertainty reduction, while also giving some insights into the accuracy of the approximate conditional mean $(\m_{0\vert\datasub},\aux_{0\vert\datasub})$ and covariance $\mathcal{C}_{\m,\aux\vert\datasub}$ (see  (\ref{eq: BAEcm}) and (\ref{eq: BAEvar}) respectively) relative to those computed using MCMC. For the MCMC approach, we use a single chain of four million samples and discard the first four hundred thousand as {\em burn-in}. The results of the reference case, including those found using the linear BAE-based approach, are shown in Figure~\ref{fig:darcy5} for $\m$ and in Figure~\ref{fig:darcy3} for $\aux$ (and $\psi$). 

We generally observe that the (approximate) posterior computed using the linear BAE approach has higher levels of uncertainty (in $\m$, $\aux$ and $\psi$) compared to respective posteriors computing using MCMC. This is expected however, as a fundamental feature of the BAE framework is the incorporation of  model errors/uncertainties (induced here by the use of the linear(-ized) surrogate model), which generally increases posterior uncertainty in the parameters. To investigate the posterior uncertainty of $\aux$ we compute the (sample-based) posterior covariance based on 10,000 posterior samples of $\psi$ (one batch from the MCMC-based posterior, and one batch for the linear BAE-based posterior) and computing $\aux=\Phi(\psi)$. It is worth noting that due to the nature of the relationship between $\psi$ and $\aux$, namely via $\Phi$, the posterior uncertainty of $\aux$ found using the linear BAE-approach is dependent on the data. This is evident by the increased uncertainty towards the interfaces of the conditional mean of $\aux$ in Figure~\ref{fig:darcy3}.

\begin{figure}[t!]
\centering
\begin{tikzpicture}

\node[inner sep=0pt] (a) at (0,0)
{\includegraphics[height=.2\textwidth]{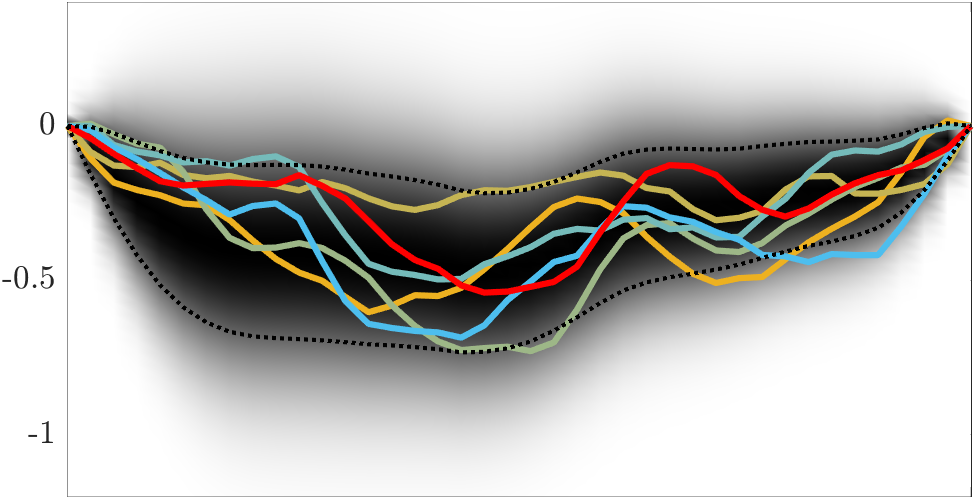}};
\node[inner sep=0pt] (a) at (7,0)
{\includegraphics[height=.2\textwidth]{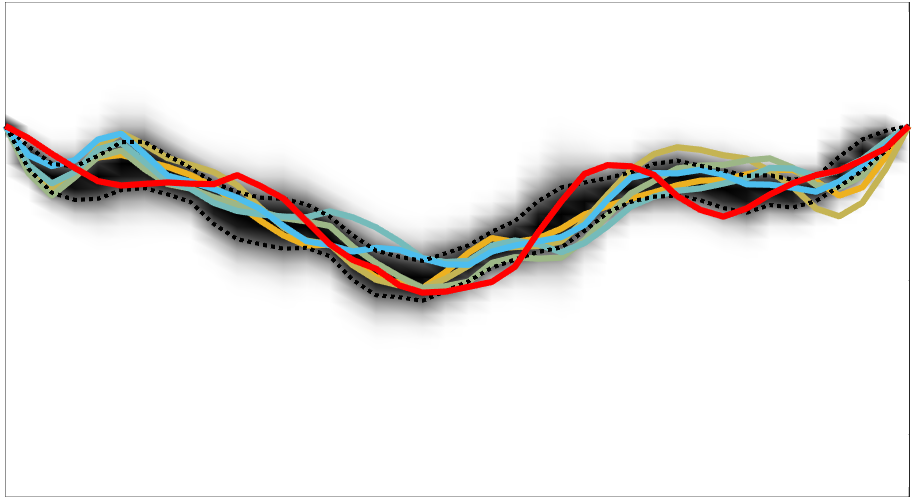}};

\end{tikzpicture}
\caption{Comparison of the marginal distributions of $\m$ for the subsurface flow problem. On the left we show samples of $\m$ from the (approximate) posterior found using the linear(-ized) approach while on the right we show samples of $m$ from the posterior found using pCN MCMC. In both figures the grey-scale represents the respective posterior density and the black dotted lines the respective posterior plus/minus (marginal) one standard deviation.}
\label{fig:darcy5}
\end{figure}

\begin{figure}[h!]
\centering
\begin{tikzpicture}
\node[inner sep=0pt] (a) at (0,0.04)
{\includegraphics[height=.205\textwidth]{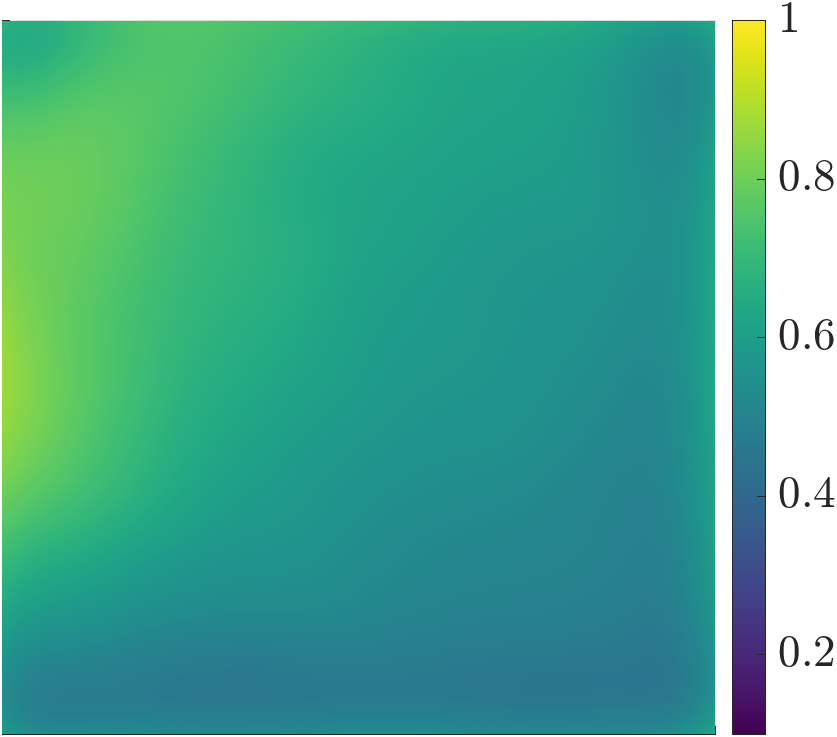}};
\node[inner sep=0pt] (a) at (4,0)
{\includegraphics[height=.2\textwidth]{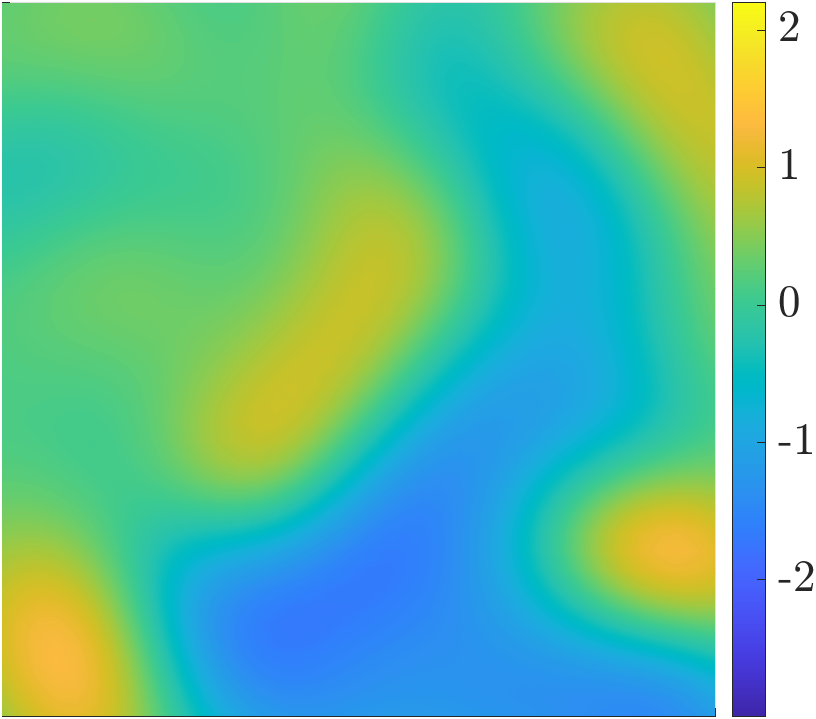}};
\node[inner sep=0pt] (a) at (8,0)
{\includegraphics[height=.2\textwidth]{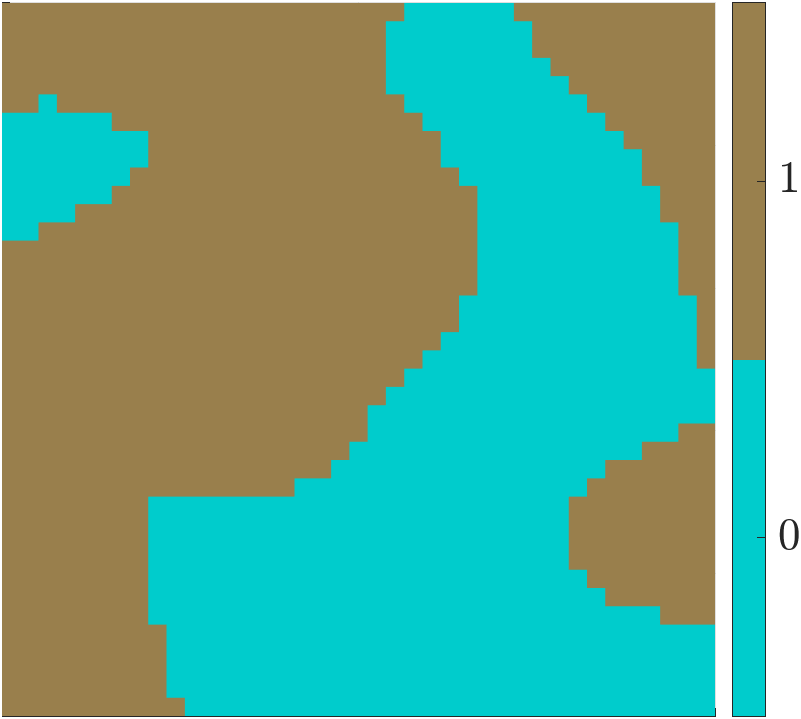}};
\node[inner sep=0pt] (a) at (12,0.04)
{\includegraphics[height=.205\textwidth]{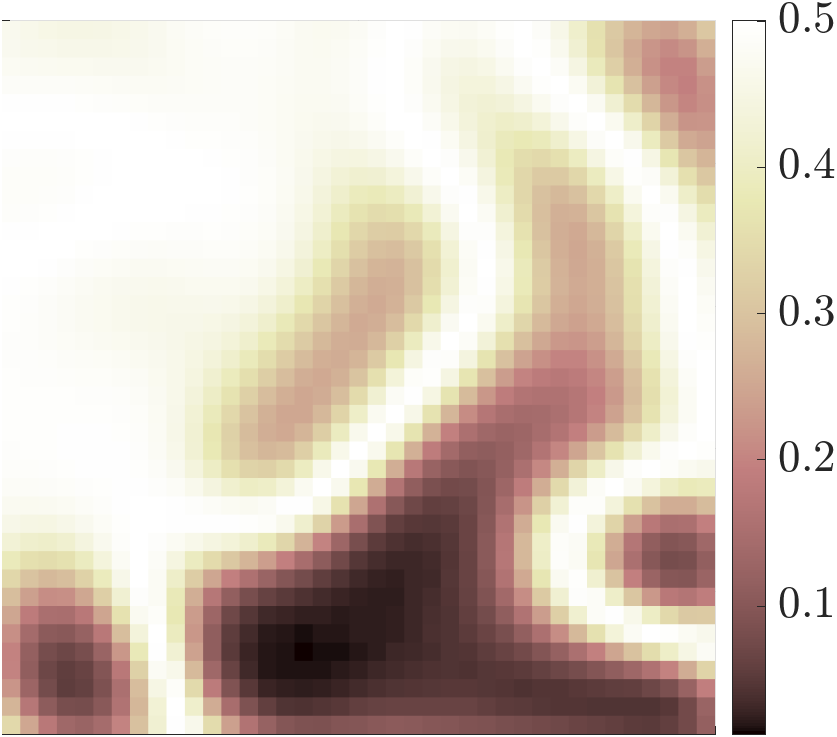}};
\node[inner sep=0pt] (a) at (0,-4)
{\includegraphics[height=.2\textwidth]{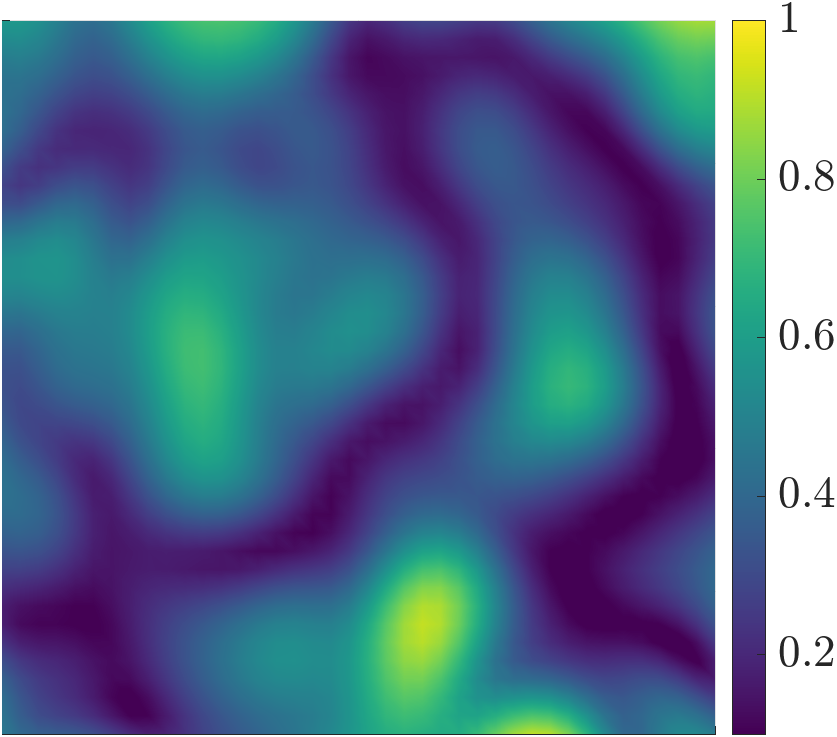}};
\node[inner sep=0pt] (a) at (4,-4)
{\includegraphics[height=.2\textwidth]{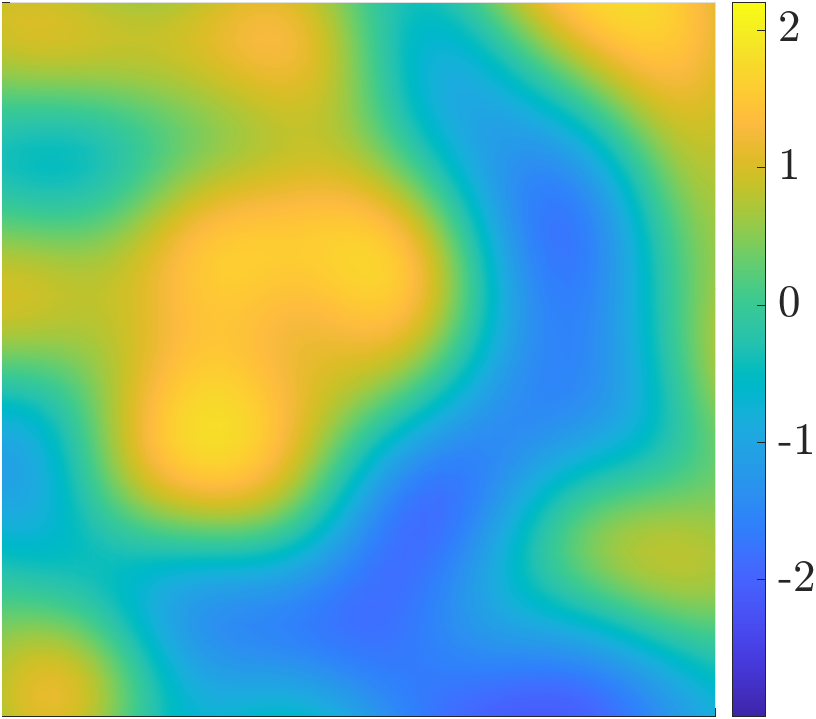}};
\node[inner sep=0pt] (a) at (8,-4)
{\includegraphics[height=.2\textwidth]{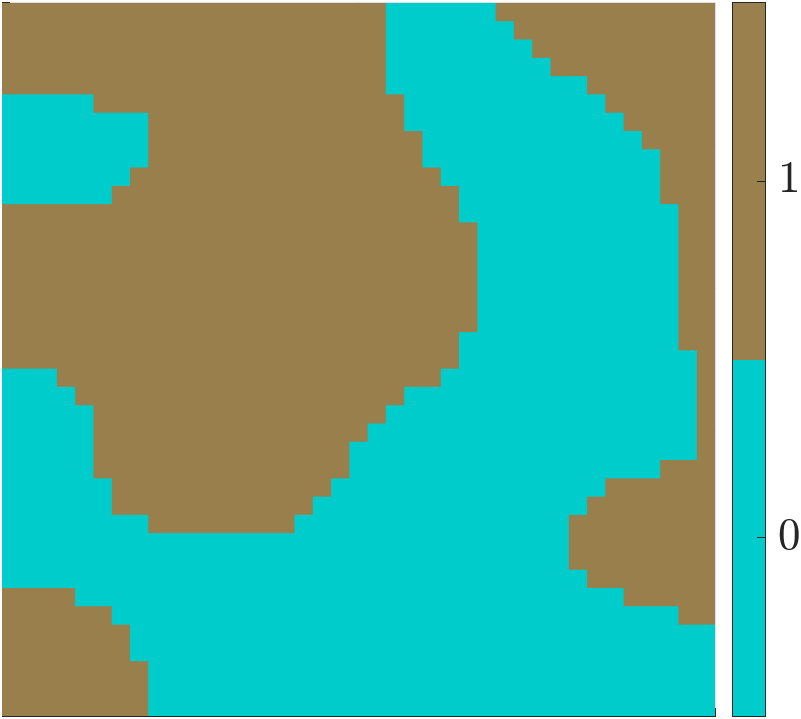}};
\node[inner sep=0pt] (a) at (12,-4.04)
{\includegraphics[height=.205\textwidth]{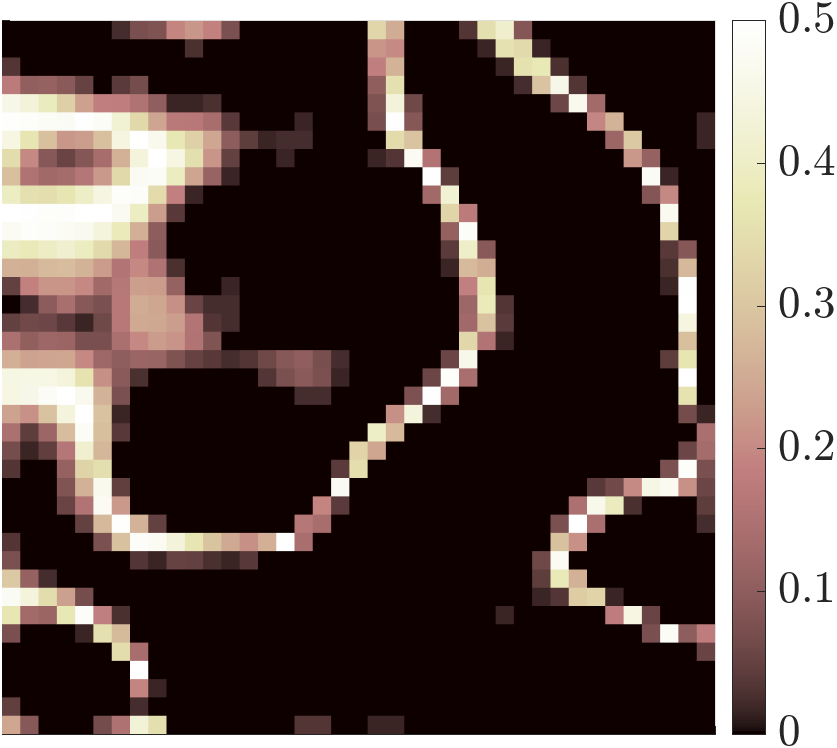}};
\end{tikzpicture}
\caption{Comparison of the marginal distributions of $\psi$ and $\aux$ for the subsurface flow problem. On the top row from left to right are shown the the marginal posterior standard deviation of $\psi$, the conditional mean estimate for $\psi$, the push-forward of conditional mean of $\psi$, \ie, $\phi(w_{\ast\vert\data})$, and the marginal posterior standard deviation of $\aux$ computing using the linear(-ized) approach, while on the bottom we show the respective quantities computed using pCN MCMC.}
\label{fig:darcy3}
\end{figure}
Our main focus in this work is concerned with reductions and measures of posterior uncertainty rather than posterior estimates themselves. However, as expected, the conditional mean estimates found MCMC-based approach do more closely resemble the true parameters compared to those found using the linear BAE-based approach (see Figure~\ref{fig:darcy3} and  Figure~\ref{fig:darcy5}).

\subsection{Optimal experimental design results}

We now consider carrying out the problem of optimal sensor selection using the proposed linearize-then-optimize OED procedure. Specifically, we consider the problem of computing the optimal 20 sensor locations to measure the pressure $u$ for each of the three problems listed above. 

To illustrate the effectiveness of the proposed approach to OED,  we compare the optimal design found to randomly chosen designs. Specifically, for a given number of sensors ${\nChosen}\in\{1,2,\dots,20\}$ we sample (without replacement) 100 random sensor configurations. The design criterion (trace of the relevant joint or marginal (approximate) posterior covariance operator) is then evaluated for each of the random sensor configurations. It should be noted that the optimal designs computed using the linear surrogate may not be optimal for the true posterior, even if uncertainty is accounted for. To this end, for the joint OED problem (Problem 1) we also investigate the effectiveness of our computed designs in minimizing $\mathbb{E}_{\data \vert \weight}\left[\trace(\mathcal{C}_{\both\vert\datasub,w}(\weight,\data))\right]$, where $\mathcal{C}_{\both\vert\datasub,w}$ is the ``true'' posterior covariance operator obtained using the model~\eqref{eq: main1} with the accurate PTO map $\afwd$.

The results for Problem 1 are shown in Figure~\ref{fig:DarcyLocs1}. The optimal design found using the proposed linearize-then-optimize approach as well a comparison of this optimal design to randomly selected designs is presented.
We see that the design found using our proposed approach outperforms  the random designs. This becomes more pronounced as the number of sensors increases up to 20\footnote{If the number of sensors were to grow toward the total possible number of sensors the difference would diminish.}. This appears to demonstrate the usefulness of the proposed method for OED. As 
alluded to above, to further investigate the applicability of our approach, for Problem 1 we also compare the performance of the designs found using the linearized approach to random designs using a SAA to $\mathbb{E}_{\data \vert \weight}\left[\trace(\mathcal{C}_{\both\vert\datasub,w}(\weight,\data))\right]$. Specifically, for comparison, 10 different samples of $(\m,\aux)$ from the joint prior are used to generate 10 sets of data at all measurement locations. For each sample we then evaluate the trace of the accurate posterior covariance operator $\mathcal{C}_{\both\vert\datasub,w}(\weight,\data)$ with a subset of data collected at: the optimal design computed using the linearize-then-optimize approach, as well as for 10 (randomly selected) different designs, all consisting of 20 sensors. Each of the (110 in total) accurate posterior covariances are computed using the pCN MCMC algorithm using a single chain of four hundred thousand samples with the first one hundred thousand discarded as burn in. Note in this case the optimal design significantly outperforms the randomly chosen designs.
\begin{figure}
\centering
\begin{tikzpicture}
\node[inner sep=0pt] (a) at (0,0)
{\includegraphics[height=.3\textwidth]{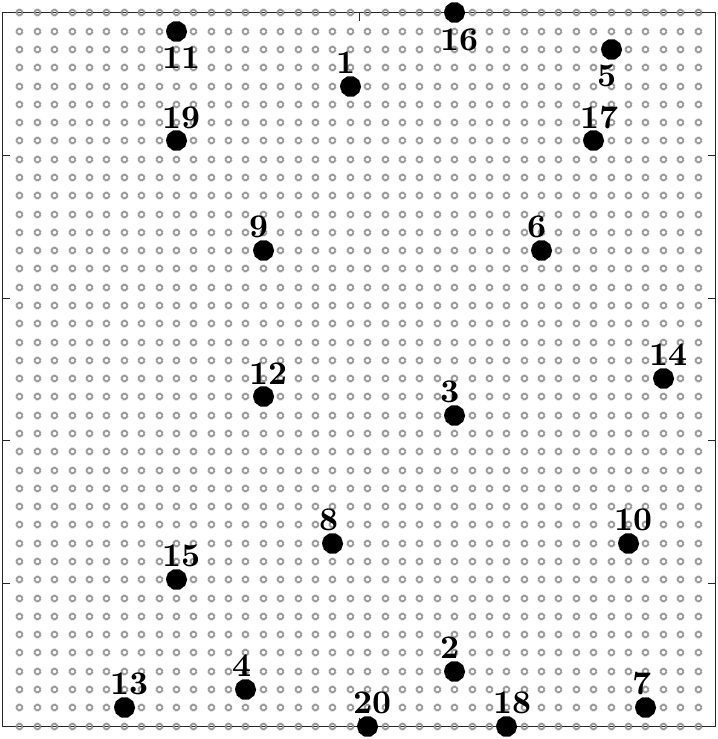}};
\node[inner sep=0pt] (a) at (6.2,0)
{\includegraphics[height=.3\textwidth]{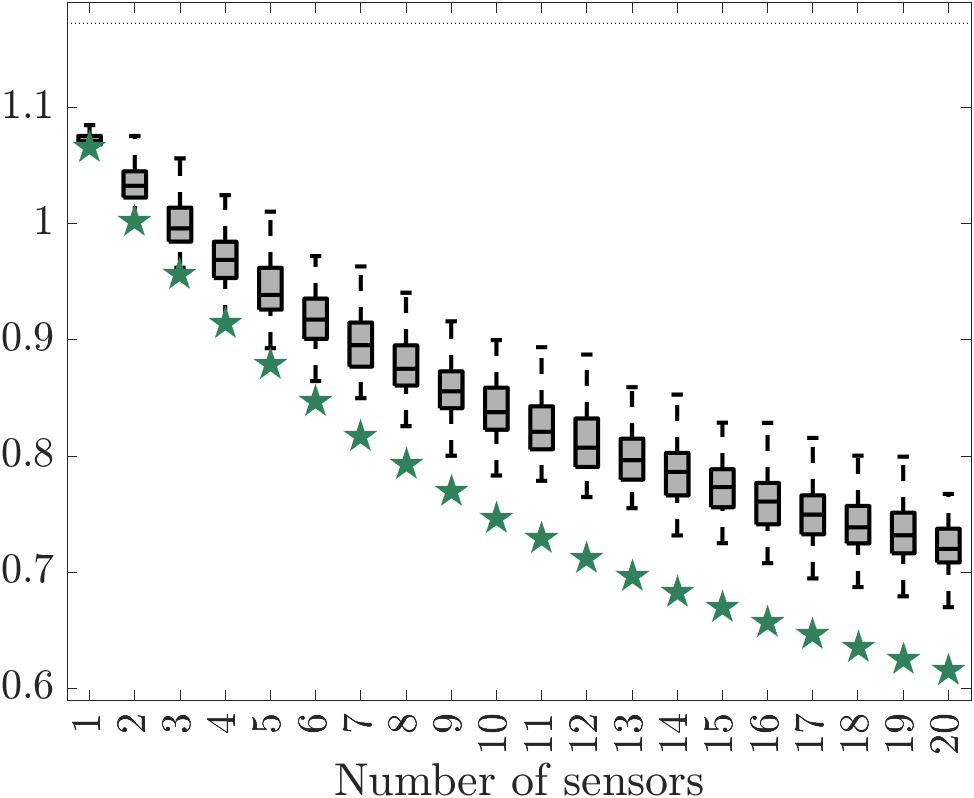}};
\node[inner sep=0pt] (a) at (11.2,0)
{\includegraphics[height=.3\textwidth]{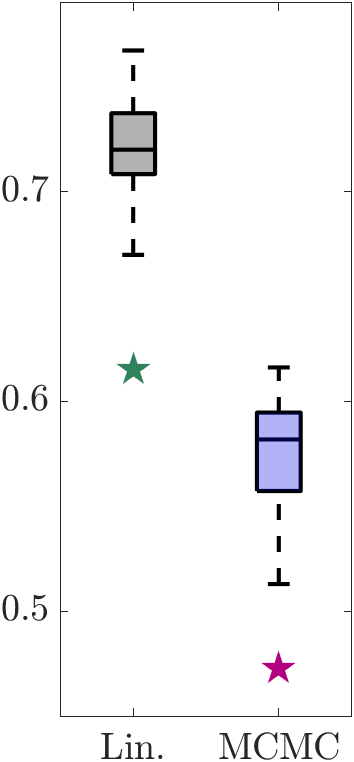}};
\end{tikzpicture}\caption{Optimal experimental design results for the the subsurface flow problem treating $\m$ and $\aux$ as inversion parameters. On the left we show the optimal sensor placements using the linearize-then-optimize approach. In the center we compare the trace of the uncertainty-aware approximate posterior covariance operator found using the optimal sensor (green stars) to  100 randomly chosen designs (black boxplot). On the right we compare the trace of the posterior covariance operator for 20 sensors using the linearization and pCN MCMC for the optimal sensor placements (green and magenta stars, respectively) and for randomly chosen designs (black and blue boxplots, respectively) and the trace of the prior covariance operator (black dotted line).}
\label{fig:DarcyLocs1}
\end{figure}

We next investigate the performance of the proposed approach on the marginal OED problem, \ie, carrying out OED for $\m$ while treating $\aux$ as an auxiliary parameter. As a comparison, we also consider carrying out the uncertainty-unaware formulation of the marginal OED problem, i.e., ignoring the additional uncertainty due to unknown auxiliary parameter.
The results for the marginal OED problem are shown in Figure~\ref{fig:DarcyLocs3}. There is a significant difference in the optimal designs found using the uncertainty-aware approach and the uncertainty-unaware approach. This can most likely be attributed to large uncertainty in the approximation errors resulting from marginalizing over $\aux$ (which has significantly larger uncertainty than $\m$). While the optimal design found using the proposed uncertainty-aware approach are reasonably distributed throughout the domain, the optimal design found using the uncertainty-unaware approach are localized near the top boundary. This is to be expected: the boundary flux $\m$ is defined over the top boundary $\Gamma_{\rm T}$ only, thus when ignoring the additional uncertainty induced by $\aux$ it is natural to measure near or on  $\Gamma_{\rm T}$ to reduce the uncertainty in $\m$.
When comparing to random designs, we see that both the uncertainty-aware and uncertainty-unaware designs perform well (using the uncertainty-aware formulation of the posterior), however the uncertainty-aware design outperforms the corresponding uncertainty-unaware design for all cases considered.

\begin{figure}[ht]
\centering
\begin{tikzpicture}
\node[inner sep=0pt] (a) at (0,0)
{\includegraphics[height=.3\textwidth]{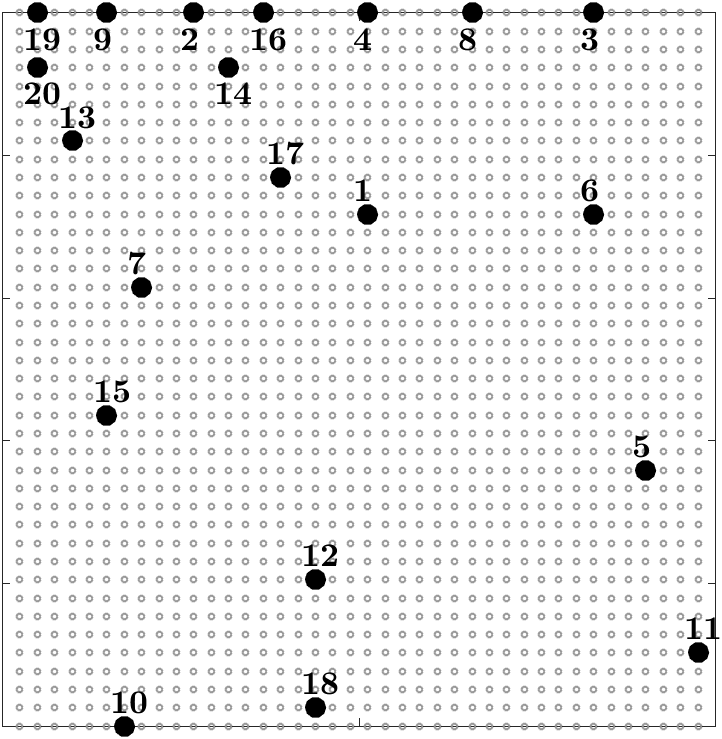}};
\node[inner sep=0pt] (a) at (4.9,0)
{\includegraphics[height=.3\textwidth]{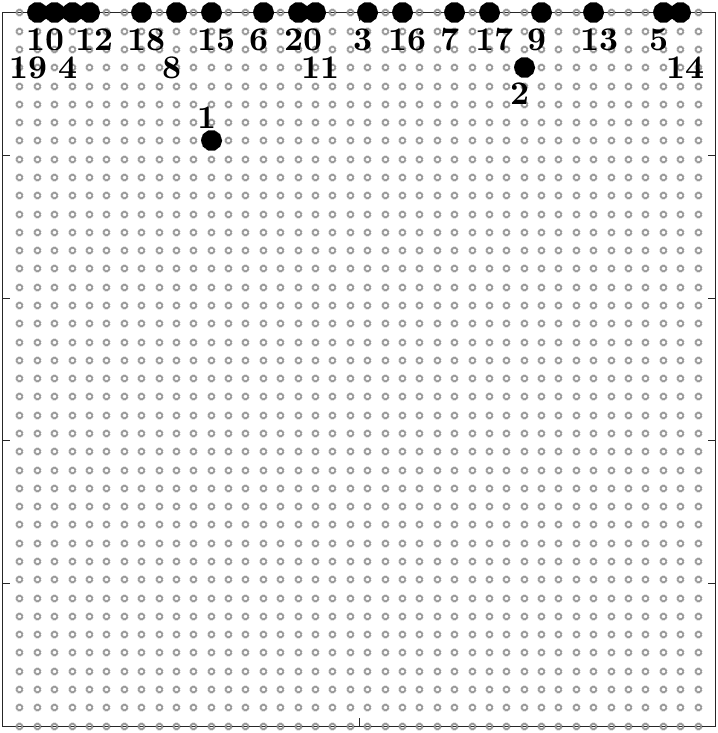}};
\node[inner sep=0pt] (a) at (10.4,0)
{\includegraphics[height=.3\textwidth]{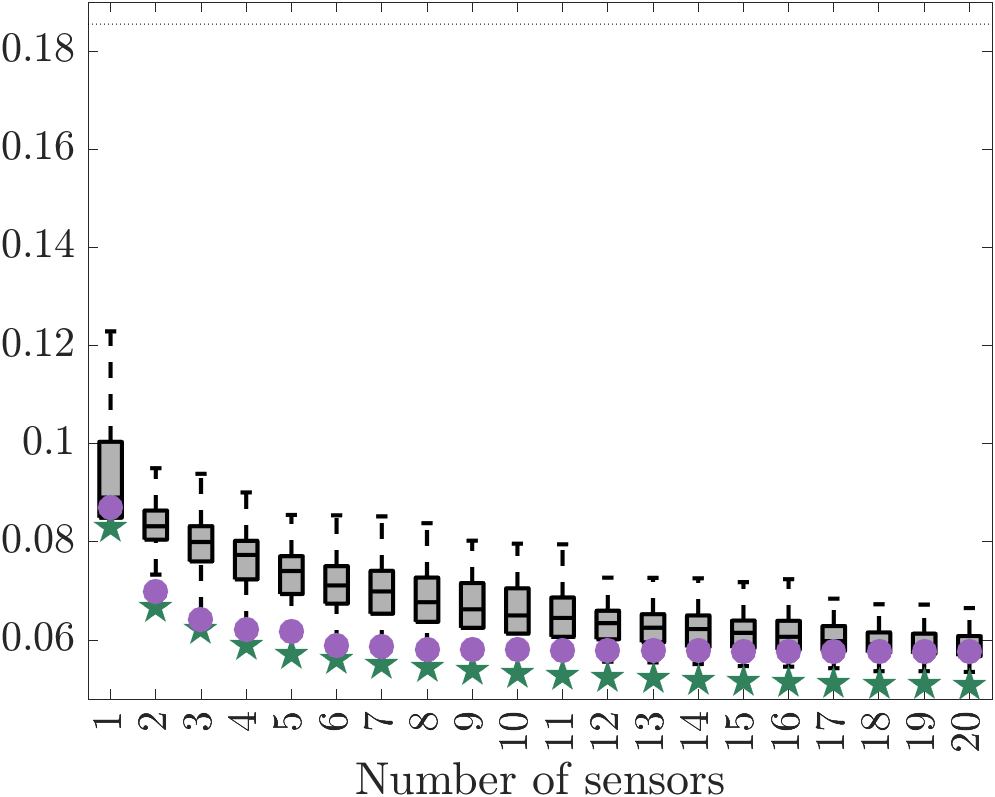}};
\end{tikzpicture}\caption{Marginal optimal experimental design results for the subsurface flow problem treating only the boundary flux $\m$ as the inversion parameter. On the left is the optimal sensor locations found using the uncertainty-aware approach, while in the center is the optimal sensor locations found using the uncertainty-unaware approach. On the right we compare the trace of the uncertainty-aware posterior covariance operator found using the optimal sensor locations found using the uncertainty-aware(green stars) and uncertainty-unaware (purple circles) to 100 randomly chosen designs (black boxplot) and the trace of the (marginal) prior covariance operator (black dotted line).
}
\label{fig:DarcyLocs3}
\end{figure}


\section{Numerical example 2: tsunami detection problem}\label{sec5}
In our second example, we aim to find optimal designs for tsunami source reconstruction in the deep ocean. 
Propagation of earthquake-induced tsunami waves in a two-dimensional spatial domain $\Omega \subset \mathbb{R}^2$ is commonly modeled using the shallow water equations (SWE)~\cite{leveque:tsunami,tong:extreme}. The SWEs are a nonlinear hyperbolic system of depth-averaged conservation laws used to model gravity waves and are well-suited for simulating tsunami waves due to their characteristically long wavelengths relative to water depth. For tsunamis arising from an instantaneous change to the ocean floor, \ie, a \emph{bathymetry} change, the shallow water equations describing the changing water depth $h(x,y,t)$ (defined as the height of the water above the ocean floor) and fluid flows in the $x$ and $y$ directions ($u(x,y,t)$ and $w(x,y,t)$ respectively) at any point $(x,y,t) \in \Omega \times (0,T]$, are
\begin{align*}
h_t + u_x + w_y &= 0, \\
u_t + \left(\frac{u^2}{h} + \frac{1}{2}gh^2\right)_x + \left(\frac{uw}{h} \right)_y &= -g h B_x, \\
v_t + \left(\frac{uw}{h}\right)_x + \left(\frac{w^2}{h}+\frac{1}{2}gh^2 \right)_y &= -g h B_y,
\end{align*}
where $0 > B(x,y) \in \mathcal{B}$ is the post-earthquake bathymetry and $g$ is the gravitational acceleration. It is assumed that the ocean is initially at rest, \ie, $h(x,y,0) = -B_R(x,y)$ and $h_t(x,y,0) = 0$ where $B_R(x,y)$ is the pre-earthquake bathymetry. 

\subsection{Modeling the bathymetry change using the Okada model}\label{subsec:Okada}

Our target example are tsunamis caused by suboceanic earthquakes. The Okada model~\cite{okada:surface} is commonly utilized to model the relationship between slips at fault plates beneath the ocean floor and seafloor deformations or bathymetry changes. Given a discretization of the fault region into a finite number of patches, the Okada model assumes the Earth behaves like a linear elastic material and provides a closed-form expression for evaluating the instantaneous bathymetry change induced by slips at these fault patches in a prescribed \emph{rake} or direction.  

Given a vector of slip magnitudes $\M{m} = [m_1,\ldots,\m_{\Npatches}] \in M$ and rakes $\boldsymbol{\aux} = [\aux_1,\ldots,\aux_{\Npatches}] \in X$ at $\Npatches \in \mathbb{N}$ fault patches, the post-earthquake bathymetry $B(x,y)$ can be defined as
\begin{equation}
B(x,y) = B_R(x,y)+\left(\mathcal{O}\M{h}(\M{\m},\boldsymbol{\aux})\right)(x,y),
\end{equation}
where the linear operator $\mathcal{O}: \mathbb{R}^{2\Npatches} \rightarrow \bothSpace$ is defined by
\begin{equation}
\left(\mathcal{O}\M{h}(\M{m},\boldsymbol{\aux})\right)(x,y) = \sum_{i=1}^{\Npatches} \mathcal{O}_i \begin{bmatrix}
    m_i \sin \aux_i \\
    m_i \cos \aux_i
\end{bmatrix},
\end{equation}
with the functions $\mathcal{O}_i = [\mathcal{O}_i^s,\mathcal{O}_i^c]$ defining the seafloor deformation induced by a slip at patch $i$.  

The Okada model makes various simplifications about the physical properties of the Earth as well as the mechanisms of the deformation (\eg, it assumes the Earth is a homogeneous isotropic elastic material and that the rupture occurs instantaneously) and thus only provides an approximation to the true bathymetry change induced by a slip at a fault. However, it is assumed to provide adequate approximations for the purposes of tsunami modeling and is often used in literature. Additionally, the scarcity of the observed data (due to financial constraints limiting the quantity of deployed deep-ocean pressure sensors) makes detailed reconstruction of the infinite-dimensional ocean deformations difficult, particularly without the use of a physically relevant prior. Parametrization of the seafloor deformation through the Okada model facilitates the use of a prior on the slip patches that results in physically realistic seafloor deformations, as will be discussed in~\Cref{subsec:LinearSWE}.

\subsection{Bayesian inversion using a linearization of the SWEs}\label{subsec:LinearSWE}

The typical goal for tsunami hazard assessment and tsunami warning systems is to estimate the tsunami-causing seafloor rupture and use the estimate for predictions and threat assessment. Time is crucial for these predictions, and using the nonlinear shallow water equations can be costly. However, away from shore in the deep ocean, the linearized SWEs (centered around the ocean at rest) provide a reasonably accurate approximation to the dynamics of propagating tsunami waves~\cite{leveque:tsunami}. This motivates the use of an affine surrogate to the PTO map obtained through a first-order Taylor expansion centered around the bathymetry of the ocean at rest, \ie, 
\begin{equation}
\data \approx \afwd(\mdisc_R,\auxdisc_R)+\afwd_{\mdisc}'\vert_{\mdisc_R,\auxdisc_R}(\mdisc-\mdisc_R)+\afwd_{\auxdisc}'\vert_{\mdisc_R,\auxdisc_R}(\auxdisc-\auxdisc_R) = \afwd(\bothdisc_R)+\mathsf{D}_{\bothdisc}\afwd(\bothdisc_R)\begin{bmatrix}\hat{\mdisc} \\ \hat{\auxdisc} \end{bmatrix}, 
\end{equation}
with $\mdisc_R = [0,\ldots,0]$, $\auxdisc_R = [\frac{\pi}{2},\ldots,\frac{\pi}{2}]$ and $\afwd_{\mdisc}'\vert_{\mdisc_R,\auxdisc_R}$, $\afwd_{\auxdisc}'\vert_{\mdisc_R,\auxdisc_R}$ denoting the derivative of $\afwd$ (with respect to $\mdisc$ and $\auxdisc$, respectively) evaluated at the fixed parameters $\mdisc_R$ and $\auxdisc_R$. 
The linearized PTO map $\mathsf{D}_{\bothdisc}\afwd(\bothdisc_R)$ mapping the slip magnitude perturbations $\hat{\mdisc} \in M$ and rake vectors $\hat{\auxdisc} \in X$ to $\nData$ incremental ocean-depth observations is obtained through: $(1)$ solution of the hyperbolic system
\begin{equation}\label{eq:2DSWE_linEq_sr}
\begin{bmatrix} 
\hat{h} \\
\hat{u} \\
\hat{w} 
\end{bmatrix}_t +  \begin{bmatrix}
0 & 1 & 0 \\
-gB_R & 0 & 0 \\
0 & 0 & 0 
\end{bmatrix} \begin{bmatrix} 
\hat{h} \\
\hat{u} \\
\hat{w}
\end{bmatrix}_x +  \begin{bmatrix}
0 & 0 & 1 \\
0 & 0 & 0 \\
-gB_R & 0 & 0 
\end{bmatrix} \begin{bmatrix} 
\hat{h} \\
\hat{u} \\
\hat{w}
\end{bmatrix}_y = \begin{bmatrix}
0 \\
gB_R\left(\mathcal{O}^s\hat{\mdisc}\right)_x  \\
gB_R\left(\mathcal{O}^s\hat{\mdisc}\right)_y 
\end{bmatrix}, 
\end{equation}
with zero initial conditions for $\hat{h},\hat{u}$ and $\hat{w}$ and $(2)$ application of an observation operator mapping the incremental state $\hat{h}$ to the spatio-temporal observations of the water depth. 
We note here that the linearization, when centered at the rest bathymetry, is invariant to changes in the rake vector, \ie, $\afwd'_{\auxdisc}\vert_{\bothdisc_R}\hat{\auxdisc} \equiv 0$. This means that simultaneous inversion for both the magnitude and rake is not possible using this linear model alone. 

Our primary focus is on tsunamis caused by large magnitude earthquakes. In particular, we focus on earthquakes with a magnitude class between $8-9$. To ensure our designs perform well for detecting such earthquakes we choose a Gaussian prior on the slip magnitudes ($\mdisc \sim \mathcal{N}(\mprdisc,\Gprdisc)$) following the procedure outlined in~\cite[Sections 2 and 5]{leveque:prior}. Regardless of the earthquake magnitude, the direction of the slip for thrust earthquakes is typically around $90^\circ$, therefore we assume a priori that $\auxdisc_i \sim \mathcal{N}(90,10)$ for each slip patch $i = 1,\ldots,\Npatches$. Some sample slip magnitudes (as well as the corresponding bathymetry changes) obtained from this tailored prior are visualized in~\Cref{fig:pr_slips}. 
\begin{figure}[ht]
\centering
\includegraphics[width=0.75\textwidth]{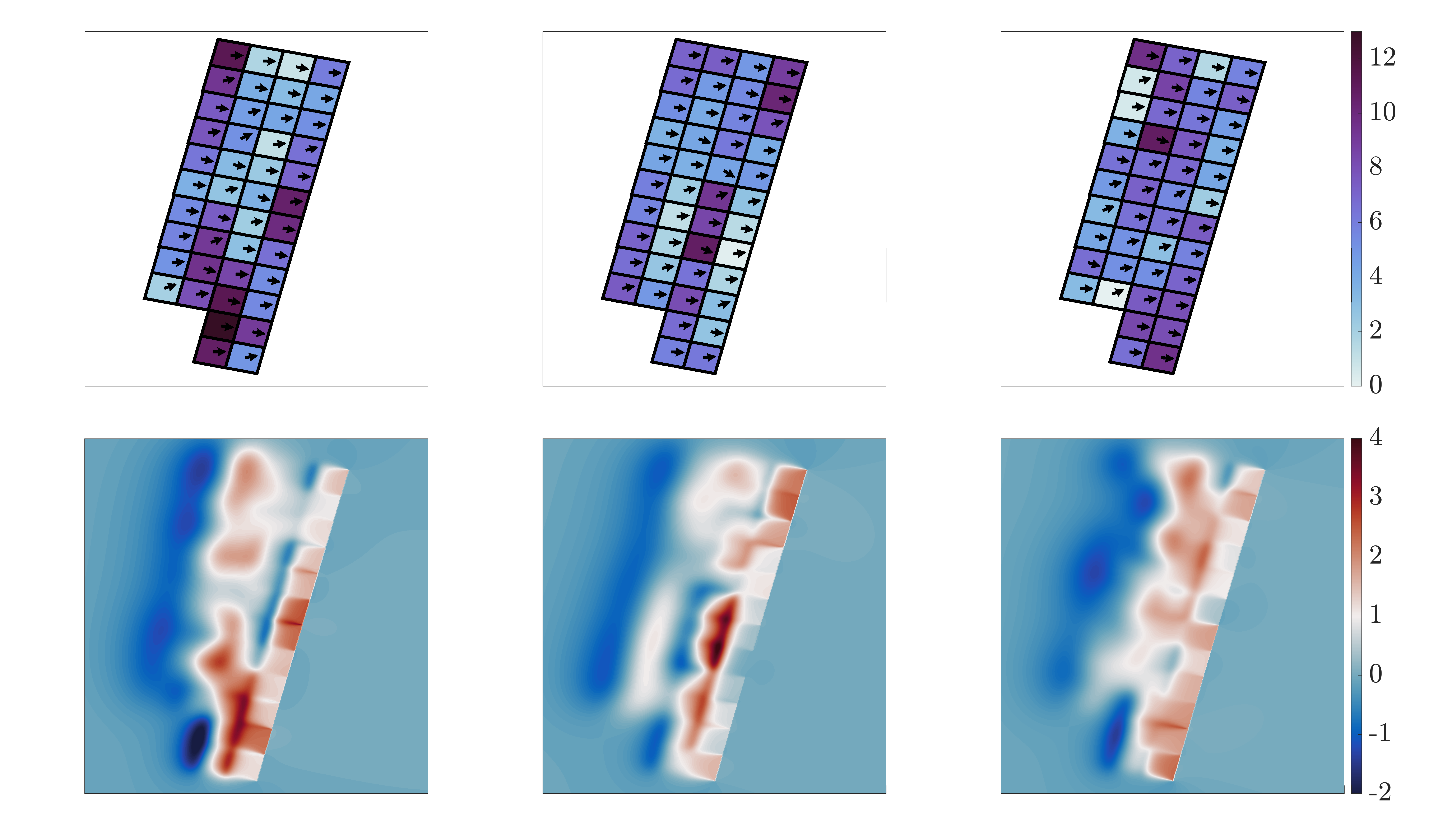}
\caption{Sample magnitudes and rakes from the prior distribution (as described in~\Cref{subsec:LinearSWE}) are visualized in the top row at the slip patches used to parameterize the tsunami-inducing bathymetry change. The bottom row visualizes the corresponding induced bathymetry change due to each sample slip.}
\label{fig:pr_slips}
\end{figure}

\subsection{Computational setup}\label{subsec:comp_setup}

The bathymetry data used for the simulations, visualized in~\Cref{fig:setup}, is from~\cite{MacInnes:japan_bathy} and can be found in the corresponding repository~\cite{LeVeque:japan_bathy}. Since the linear approximation to the shallow water equations degrades in accuracy near shore, we limit our computational domain to a rectangular region offshore as can be seen in the right image in~\Cref{fig:setup}. We assume the tsunami originates from a vertical deformation of the seafloor in this region. To simulate the bathymetry change, a potential fault region is discretized into $44$ slip patches $25$ km in length and width, each with a dip of $14^\circ$ and strike of $193^\circ$ (see~\Cref{fig:setup}). The slip patches chosen for these numerical experiments are a slightly modified subset of of the patches used in~\cite{fuji:patches}. Specifically, to suit our needs, each original patch was split into four equally-sized patches and the depths were adjusted accordingly. 

We model our ocean-floor sensors on a simplified Deep-ocean Assessment and Reporting of Tsunamis (DART) II system, which consists of a pressure sensor tethered to the ocean floor and a seasurface companion buoy equipped with satellite telecommunications capability~\cite{meinig:DARTII}. For simplicity, we assume the sensors can measure the height of the water column above them directly. We specify $214$ locations for possible data collection, assuming the sensors can only be placed in depths between $1-6$ kilometers. A visualization of these possible locations is shown on the right in Figure~\ref{fig:setup}. Since the water amplitude can be resolved with higher accuracy away from the region of maximum deformation (see, \eg,~\cite{meinig:DARTII}) we impose uncorrelated Gaussian measurement noise with zero mean and standard deviation $0.005$ meters for data collected at sensors situated away from the fault and $0.08$ meters for sensors close to the fault.
\begin{figure}[ht]
\center
\includegraphics[width=0.9\linewidth]{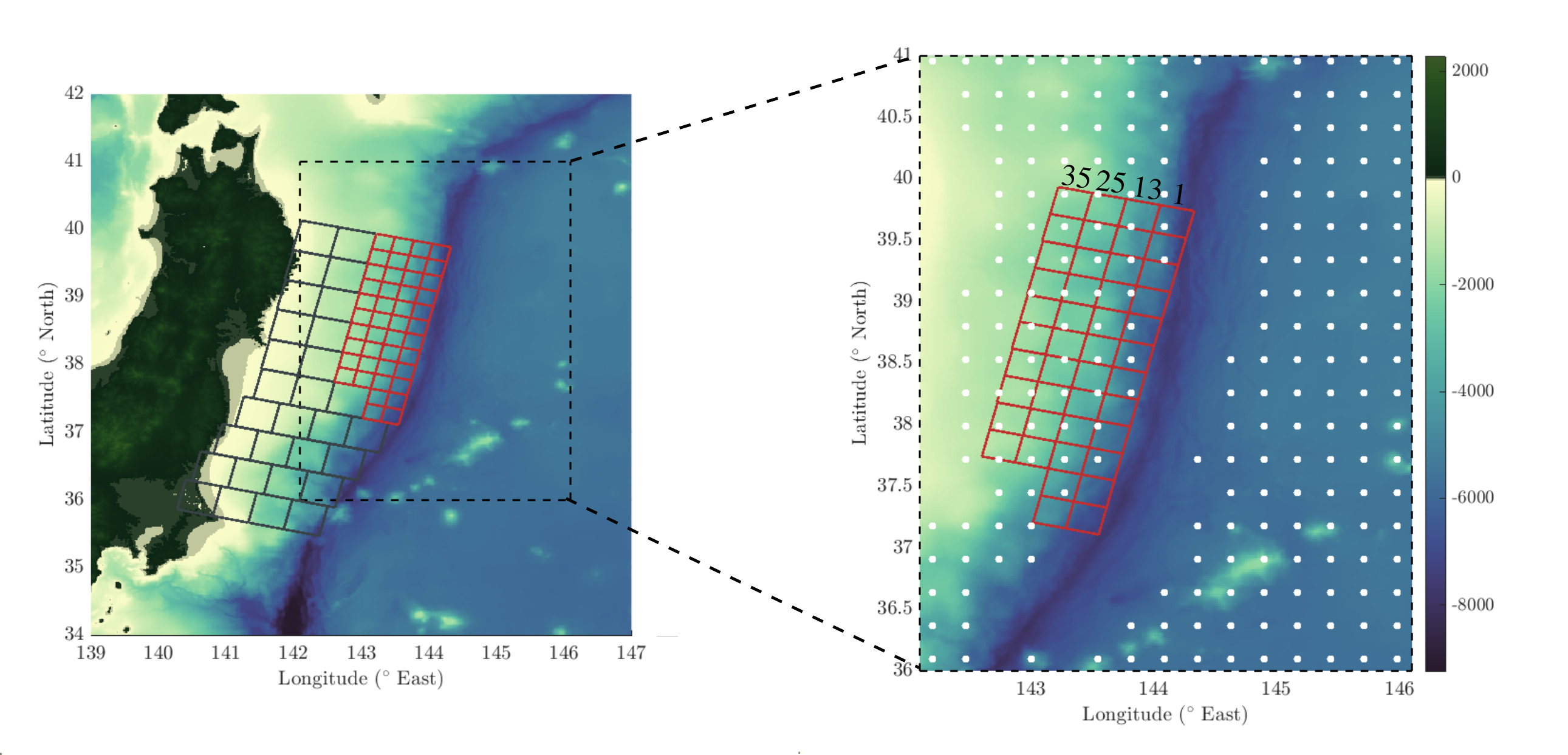}
\caption{The left figure is included for spatial reference, it depicts the topological data of the full domain. The solid grey lines indicate the boundaries of the slip patches obtained from~\cite{fuji:patches} which we have split in quarters. The boundaries of the patches used in our simulations are depicted by solid red lines. The right image shows a zoomed-in view of our computational domain with the numbering used for the patches. The white dots show potential locations for sensor placement.}
\label{fig:setup}
\end{figure}

To simulate ``event mode'' of the DART sensors, we assume each sensor produces measurements at $30$ second intervals for the first four minutes (starting at $5$ seconds after the seafloor rupture) resulting in $9$ depth readings for each sensor. For simplicity, we assume the $30$-second depth interval readings are average readings over a two second interval, \ie, for $\tau_i \in \{5,35,\ldots,245\}$, $h(\tau_i) \approx \frac{1}{2} \int_{\tau_i - 1}^{\tau_i + 1} h(t) \, dt $, which we approximate with the trapezoidal rule. 

\subsection{Results} 
In this section, we compare the optimal placement of sensors using the uncertainty-aware and uncertainty-unaware formulation of the OED problem.
The uncertainty-aware designs are obtained treating both the magnitudes $\mdisc$ and directions $\auxdisc$ at each fault patch as parameters-of-interest and accounting for uncertainty due to the use of the linear surrogate map. All the forward simulations in this section were performed using the Conservation Laws Package (Clawpack~\cite{clawpack,mandli2016clawpack}) and the GeoClaw toolbox~\cite{berger:geoclaw}. 

As in~\cref{sec4}, we again choose the zero operator for obtaining the uncertainty-aware designs. For the uncertainty-unaware designs, the linearized SWEs~\eqref{eq:2DSWE_linEq_sr} are used to approximate the propagating tsunami and model error is ignored. For both uncertainty-aware and uncertainty-unaware designs, we compute the optimal placement of the sensors using the greedy approach outlined in~\Cref{sec3 1} with $500$ samples used to approximate the BAE statistics for the uncertainty-aware model. The uncertainty-aware and uncertainty-unaware optimal placements of $20$ sensors are visualized in~\Cref{fig:tsunami_sensComp}. 
\begin{figure}[ht]
\centering
\includegraphics[width=0.8\textwidth]{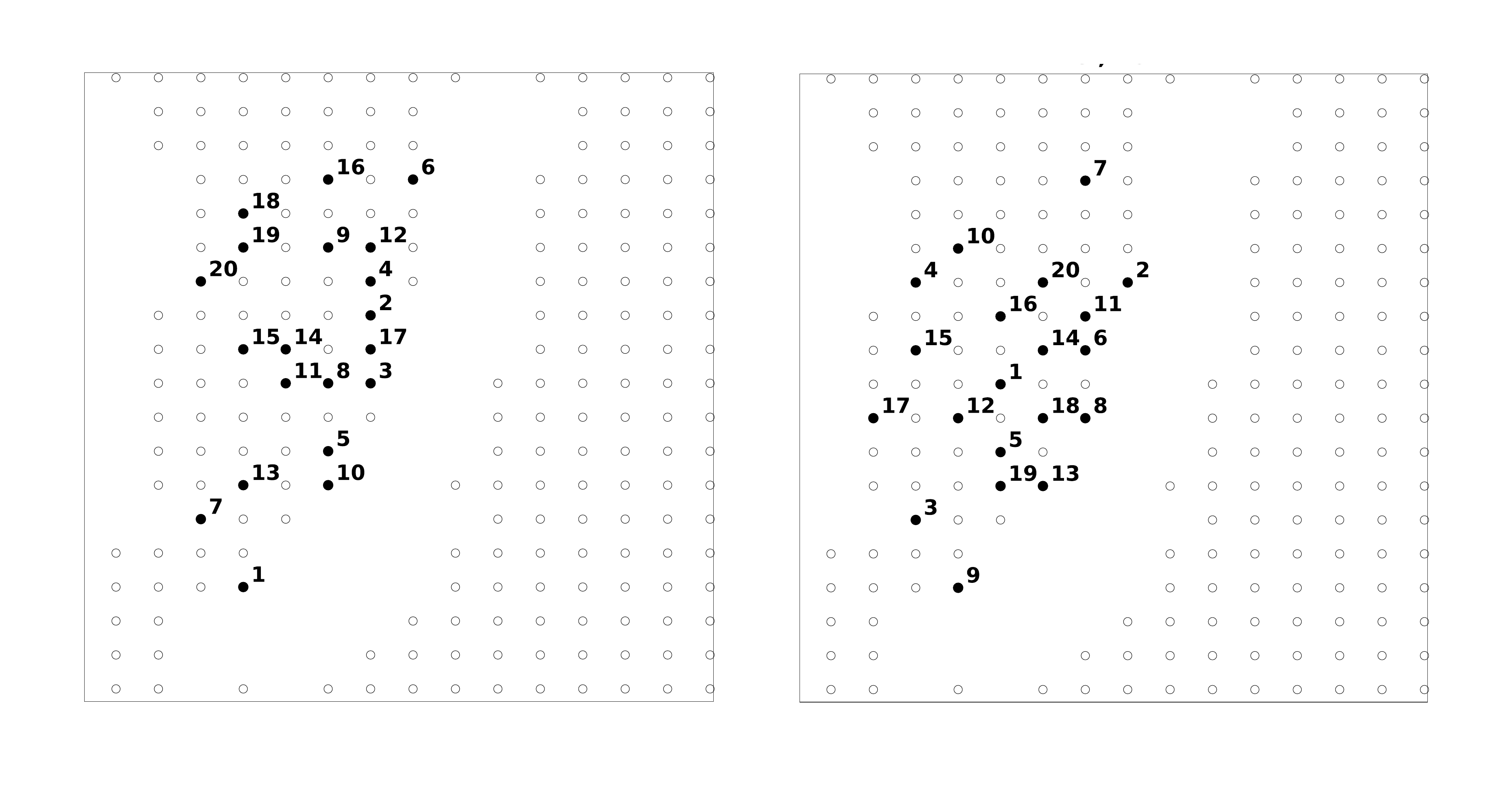}
\caption{Uncertainty-aware sensor placements for the tsunami model problem (left column), and the uncertainty-unaware optimal design (right column). In all the plots, the chosen sensors are numbered according to the order in which they were chosen using the greedy approach.}
\label{fig:tsunami_sensComp}
\end{figure}

To illustrate the effectiveness of uncertainty-aware designs compared to uncertainty-unaware designs, in~\Cref{fig:tsunami_postMean} we compare the posterior mean and standard deviation obtained using both designs with 13 sensors and the uncertainty-aware formulation of the Bayesian inverse problem. In all cases, noisy data was simulated with a randomly chosen ``true'' parameter vector $[\mdisc_{\text{true}},\auxdisc_{\text{true}}]$ and the full nonlinear SWEs. 
The designs optimized using the linearized SWE without accounting for uncertainty lead to lower posterior uncertainty in the slip magnitudes than the uncertainty-aware designs. This is not surprising --- uncertainty in the rakes can not be reduced using~\eqref{eq:2DSWE_linEq_sr}, therefore sensor placements are chosen to optimally infer the slip magnitudes. Thus, the uncertainty is higher in the slip rakes, and the overall reconstruction of the fault (and resulting bathymetry change) is worse when using the uncertainty-unaware designs. This observation is strengthened by the relative distance (measured using the Euclidean norm) from the truth to the posterior means, which is approximately: 0.44 and 0.26 for the relative error in the slip magnitude estimates using the uncertainty-unaware and uncertainty-aware designs, respectively, and 0.17 and 0.13 for the relative error in the slip rake estimates using the uncertainty-unaware and uncertainty-aware designs, respectively. 
Note that without accounting for model error in the Bayesian inverse problem, the optimal uncertainty-unaware design performs rather poorly in reconstructing the true parameter vector (see~\Cref{fig:tsunami_postMean_unaware}). 
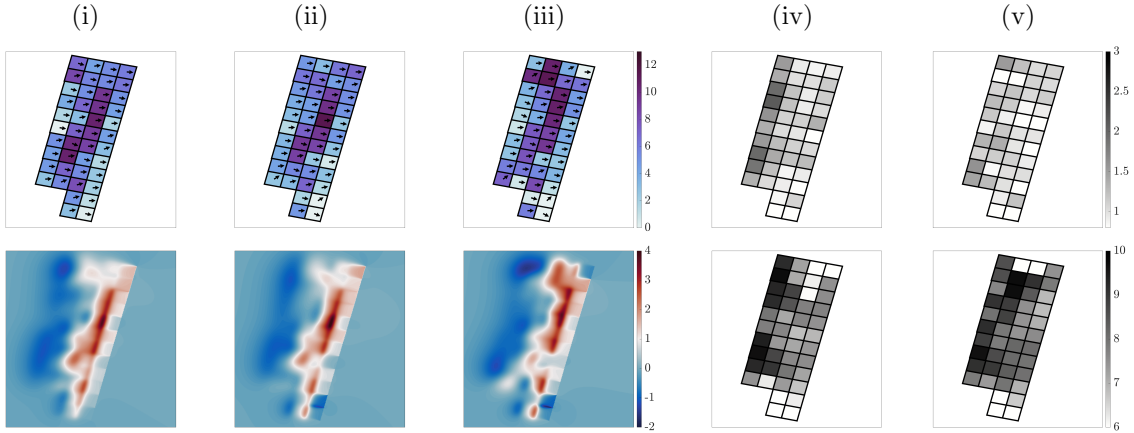
\begin{figure}[hb!]
	\centering
	\begin{tikzpicture}
        \node (img) {\includegraphics[width=\textwidth]{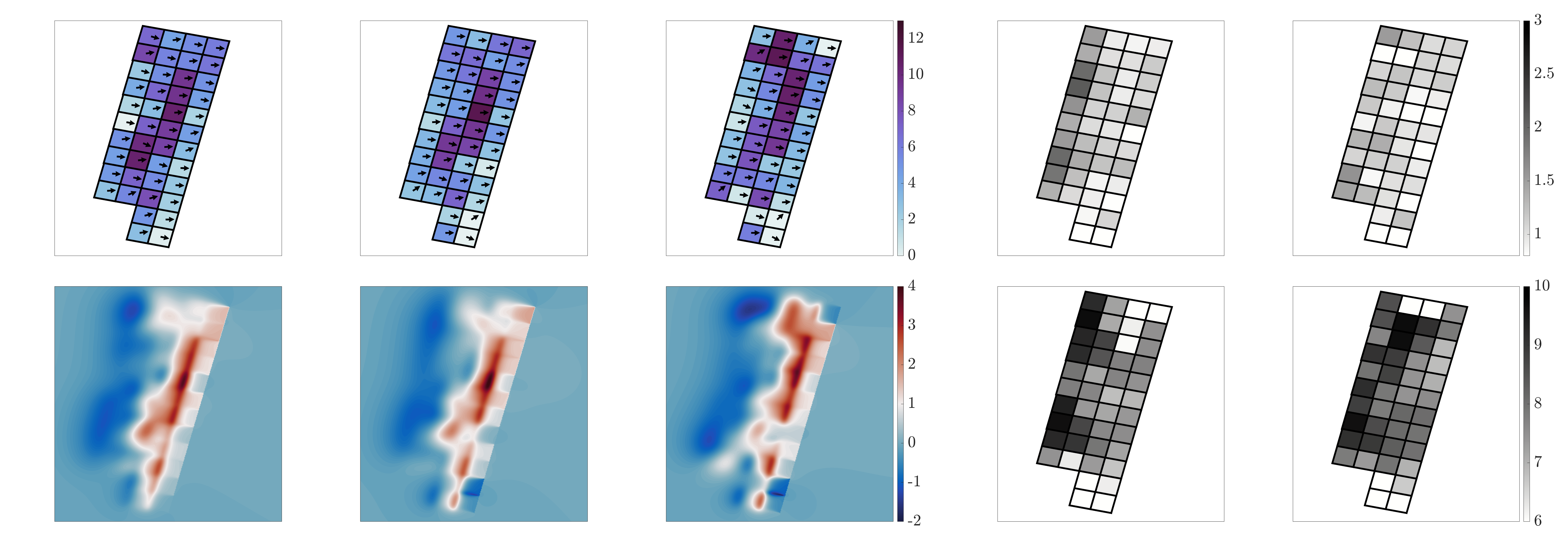}};
        \node at (-6.2,3.0) {(i)};
        \node at (-3.2,3.0) {(ii)};
        \node at (-0.1,3.0) {(iii)};
        \node at (3.2,3.0) {(iv)};
        \node at (6.2,3.0) {(v)};
	\end{tikzpicture}
 \caption{Visualization of posterior statistics using the uncertainty-aware and unaware optimal sensor placements. The slip magnitudes $\mdisc_{\text{true}}$ and rakes $\auxdisc_{\text{true}}$ as well as the corresponding bathymetry change used to simulate data is shown in the top and bottom figures of column (i) respectively. In the top row of columns (ii) and (iii) we show the posterior slip magnitudes and rakes obtained using the uncertainty-aware and uncertainty-unaware design. The corresponding average bathymetry change induced by the posterior distribution with both design choices is shown in the bottom row of columns (ii) and (iii). Columns (iv) and (v) visualize the posterior standard deviations (obtained using the uncertainty-aware and uncertainty-unaware designs respectively) in the slip magnitude (top row) and rakes (bottom row).}
 \label{fig:tsunami_postMean}
\end{figure}
\begin{figure}[hbt!]
	\centering
	\begin{tikzpicture}
        \node (img) {\includegraphics[width=0.45\textwidth]{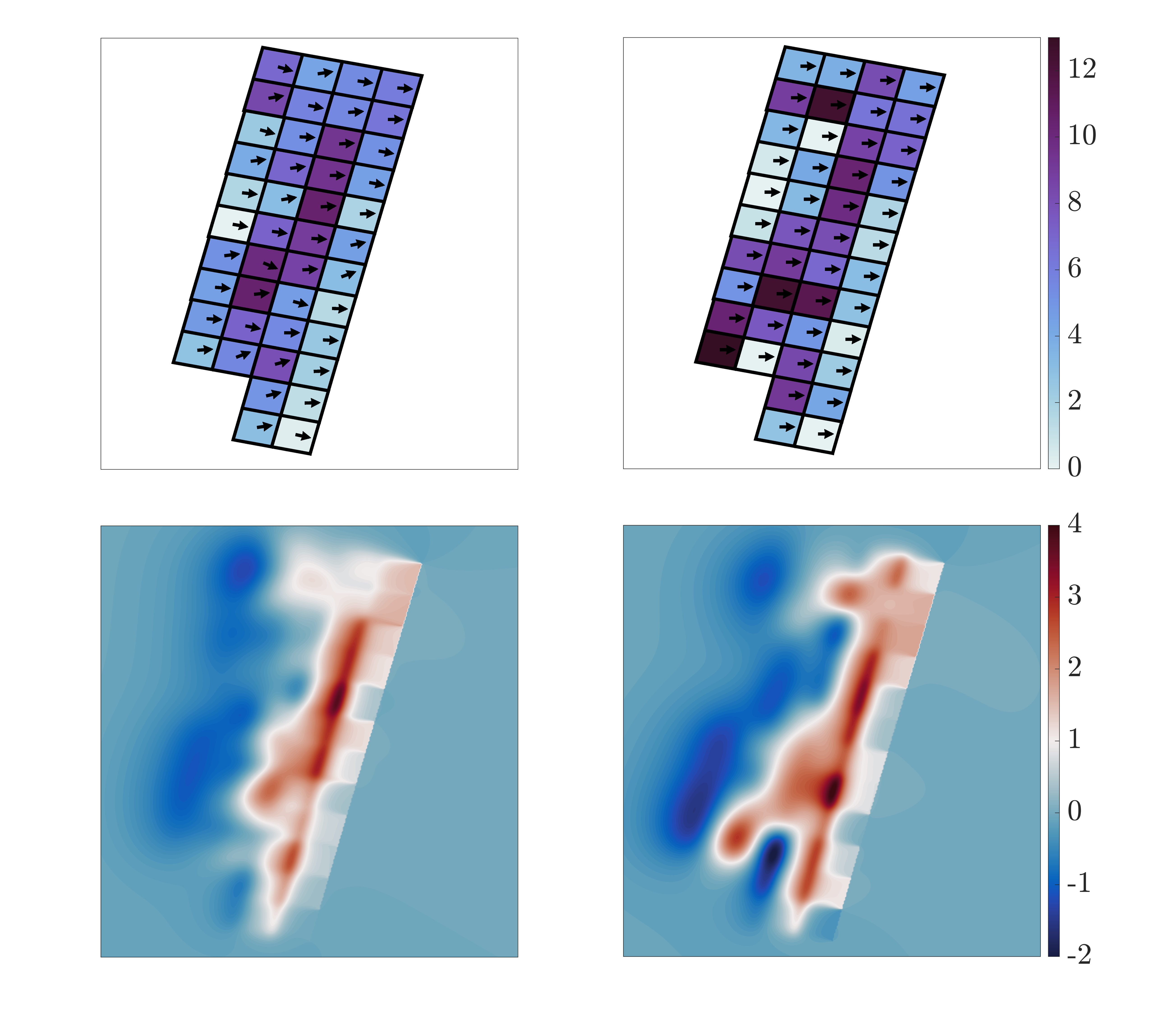}};
        \node at (-1.8,3.4) {(i)};
        \node at (1.6,3.4) {(ii)};
	\end{tikzpicture}
	\caption{Visualization of the posterior mean using the uncertainty-unaware optimal sensor placement. The slip magnitudes $\mdisc_{\text{true}}$ and rakes $\auxdisc_{\text{true}}$ as well as the corresponding bathymetry change used to simulate data is shown in the top and bottom figures of column (i) respectively. In the top row of column (ii) we show the estimated slip magnitudes and rakes for each slip patch. The corresponding average bathymetry change induced by the posterior distribution is shown in the bottom row of column (ii). The posterior was obtained using the linearized SWEs without accounting for model error.}
	\label{fig:tsunami_postMean_unaware}
\end{figure}

We further evaluate the quality of the optimal uncertainty-aware designs compared to the optimal uncertainty-unaware designs and randomly chosen configurations in~\Cref{fig:tsunamis_traceComps}. For each design configuration involving $k$ sensors, we evaluate: the trace of the resulting uncertainty-aware posterior covariance matrix, and the expected relative error of the posterior mean. To obtain the latter, we fix $100$ sample parameters $\bothdisc^{(i)}$ from the prior distribution (different than those used for computing the covariance statistics required for the BAE approach) and synthesize noisy data using the nonlinear SWEs, \ie, $\data^{(i)} = \afwd(\bothdisc^{(i)})+\noise$.
The expected error in the posterior mean is then approximated using a sample average approach, $\frac{1}{100}\sum_{i=1}^{100}\frac{\|\bothdisc^{(i)}-\bothmeandisc^{\data^{(i)}}\|}{\|\bothdisc^{(i)}\|}$, where $\bothmeandisc^{\data^{(i)}}$ denotes the posterior mean corresponding to data $\data^{(i)}$. 
The uncertainty-aware designs outperform the uncertainty-unaware and random designs in reducing the posterior trace as well as the average approximation error in the posterior mean. 
Additionally, while the uncertainty-aware designs in comparisons in~\cref{fig:tsunami_sensComp} and~\cref{fig:tsunami_postMean} were obtained using error statistics computed with $500$ samples, for~\Cref{fig:tsunamis_traceComps} we computed optimal uncertainty-aware designs using $q=50,250,500,1000$ samples. While using more samples does produce more effective designs, we emphasize that reasonable designs could be obtained using as few as $50$ samples, \ie, using only $50$ runs of the costly forward model. 
\begin{figure}[ht!]
\centering
\begin{tikzpicture}
\node[inner sep=0pt] (a) at (-8.0,0)
{\includegraphics[height=.43\textwidth]{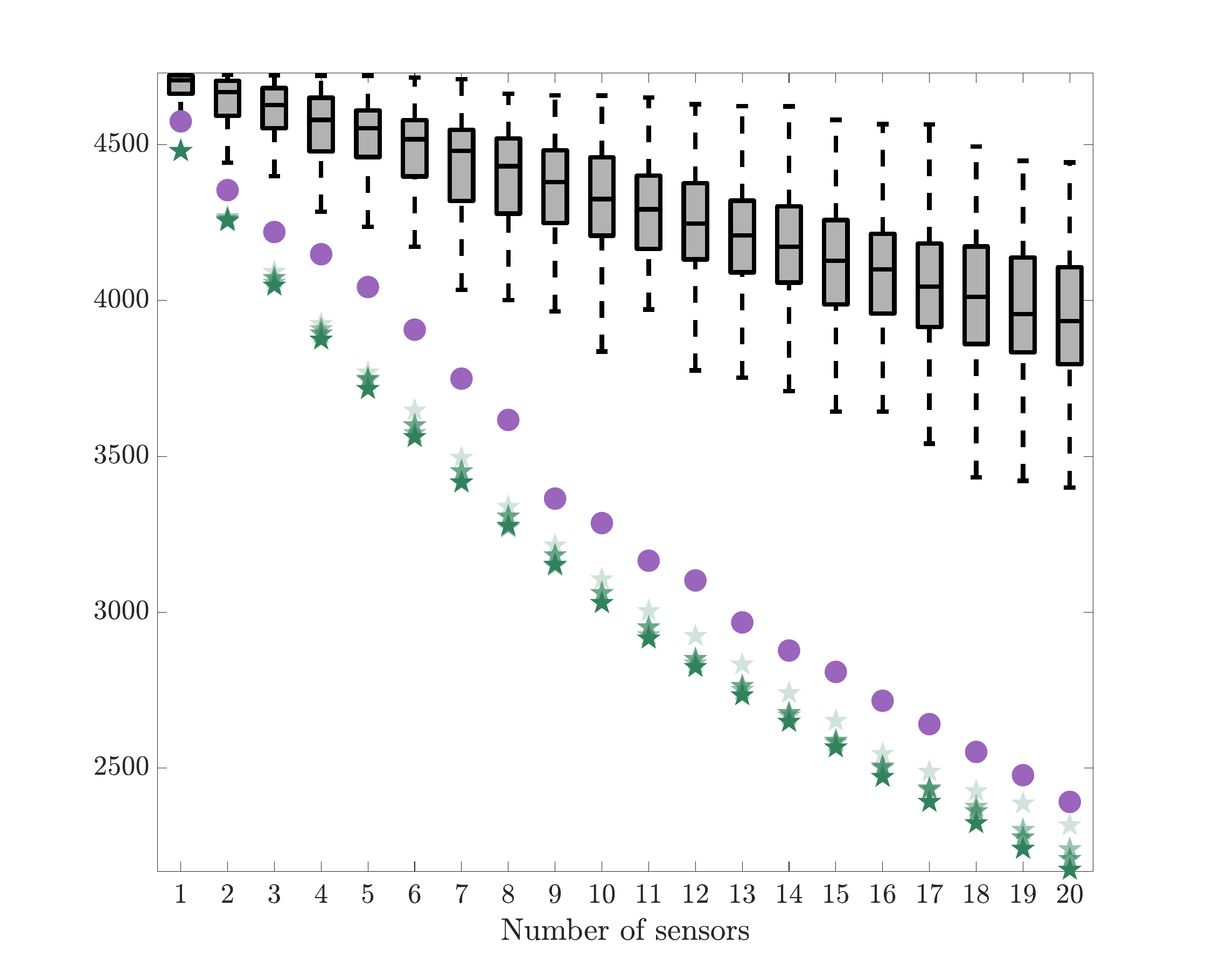}};
\node[inner sep=0pt] (b) at (0,0)
{\includegraphics[height=.435\textwidth]{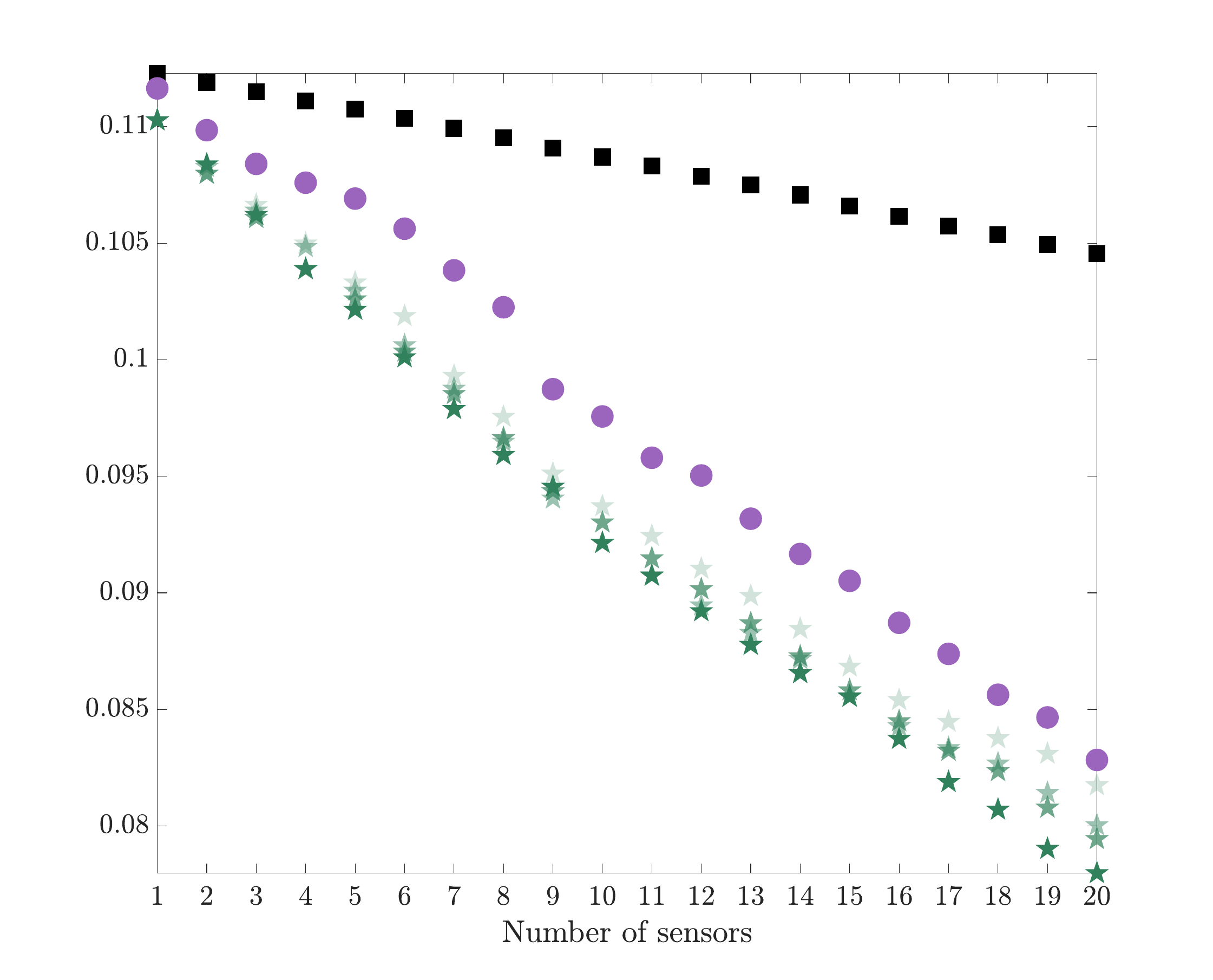}};
\end{tikzpicture}
\caption{In the left we compare the trace of the posterior covariance operator using uncertainty-aware designs (green stars), uncertainty-unaware designs (purple circles) and 100 randomly chosen designs (black boxplot). In all cases, the posterior covariance operator was computed using the zero map while incorporating model uncertainty with $1000$ samples using the BAE approach. In the right figure we compare the average relative error of the posterior mean (computed using $100$ samples $\bothdisc^{(i)} \sim \mupr$) for each design choice. The posterior mean was obtained using noisy data synthesized using $\bothdisc^{(i)}$ and the nonlinear SWE. The black squares correspond to the average over 100 randomly chosen designs. In both figures, the uncertainty aware designs were obtained using $q = 50,250,500,1000$ (visualized using green stars of increasing opacity, respectively) samples to approximate the error statistics.}
\label{fig:tsunamis_traceComps}
\end{figure}

\section{Discussion and Conclusion}\label{sec6}

We have presented a scalable approach for approximating optimal sensor placements for nonlinear Bayesian inverse problems. Our method involves replacing a nonlinear parameter-to-observable map with a linear surrogate while incorporating model error into the inverse problem with the Bayesian approximation error approach. Notably, we demonstrated that this formulation yields an approximation to the A-optimality criterion that is asymptotically independent of the specific linearization choice. This result enables a derivative-free approach to linearized OED. Through two numerical examples, we illustrated the efficacy of sensor placements obtained with our method using the zero map as a surrogate to the accurate forward dynamics.

The results in this article point to several possibilities for future work. Firstly, we reiterate that equivalence of the A-optimal designs under different linearizations can only be guaranteed asymptotically. While the design comparisons in Sections~\Cref{sec4} and~\Cref{sec5} indicate that the number of samples used in our examples were sufficient in yielding effective designs, simulating tens-of-thousands of data samples using the accurate model may be prohibitively expensive for certain problems. To accommodate such cases, exploring the use of well-chosen control variates to expedite convergence will be crucial to reduce the number of samples needed. 

Additionally, a key aspect of our approach is that it relies solely on sample parameter and data pairs. While the ``accurate'' data used in our examples was synthetic, this suggests the potential of applying our method for model-free OED. That is, in situations where the accurate forward model is unknown, our approach could be utilized using experimental data for fixed parameter choices. Lastly, another intriguing direction is a study of how well our designs perform in minimizing the uncertainty in the ``true'' posterior. We provide a heuristic comparison for the subsurface flow example in~\Cref{sec4}, however a more rigorous study would provide insight into the limitations of linearization-based approximations to optimal designs. 

\section*{Acknowledgments} 
The work of KK has been partially funded by Carl-Zeiss-Stiftung through the project ``Model-Based AI: Physical Models and Deep Learning for Imaging and Cancer Treatment''. The authors thank Alen Alexanderian, Alex de Beer, Oliver Maclaren, and Georg Stadler for helpful discussions and their valuable comments about this manuscript.

\bibliographystyle{siam}
\bibliography{sn-bibliography}

\end{document}